\documentclass[11pt]{article}
\usepackage{amsmath,amsthm,amssymb}
\usepackage{fancyhdr}
\usepackage[british]{babel}
\usepackage{geometry,mathtools}
\usepackage{enumitem}
\usepackage{algpseudocode}
\usepackage{dsfont}
\usepackage{centernot}
\usepackage{xstring}
\usepackage{colortbl}
\usepackage[section]{algorithm}
\usepackage{graphicx}
\usepackage[font={footnotesize}]{caption}
\usepackage[usenames,dvipsnames,table]{xcolor}
\usepackage{pstricks,tikz}
\usepackage[h]{esvect}
\usepackage[
		bookmarksopen=true,
		bookmarksopenlevel=1,
		colorlinks=true,
		linkcolor=darkblue,
        linktoc=page,
		citecolor=darkblue,
]{hyperref}
\usepackage{bbm}
\usepackage{appendix}

\numberwithin{equation}{section}

\definecolor{darkblue}{rgb}{0,0,0.5}

\newdimen\margin
\def\textno#1&#2\par{
   \margin=\hsize
   \advance\margin by -4\parindent
          \setbox1=\hbox{\sl#1}
   \ifdim\wd1 < \margin
      $$\box1\eqno#2$$
   \else
      \bigbreak
      \hbox to \hsize{\indent$\vcenter{\advance\hsize by -3\parindent
      \it\noindent#1}\hfil#2$}
      \bigbreak
   \fi}

\newtheorem{theorem}[algorithm]{Theorem}
\newtheorem{prop}[algorithm]{Proposition}
\newtheorem{lemma}[algorithm]{Lemma}
\newtheorem{cor}[algorithm]{Corollary}

\theoremstyle{definition}

\newtheorem{conj}[algorithm]{Conjecture}

\newtheorem{remark}[algorithm]{Remark}

\def\lateproof#1{\removelastskip\penalty55\medskip\noindent\begin{stepenv}\end{stepenv}{\bf Proof of #1. }} 

\def\noproof{{\unskip\nobreak\hfill\penalty50\hskip2em\hbox{}\nobreak\hfill%
       $\square$\parfillskip=0pt\finalhyphendemerits=0\par}\goodbreak}
\def\endproof{\noproof\bigskip}

\def\claimproof{\removelastskip\penalty55\medskip\noindent{\em Proof of claim: }}
\def\noclaimproof{{\unskip\nobreak\hfill\penalty50\hskip2em\hbox{}\nobreak\hfill%
       $\diamond$\parfillskip=0pt\finalhyphendemerits=0\par}\goodbreak}

\def\endclaimproof{\noclaimproof\medskip}

\newcounter{stepenv}
\newenvironment{stepenv}[1][]{\refstepcounter{stepenv}}{}

\newcounter{step}[stepenv]

\newcounter{substep}[step]
\renewcommand{\thesubstep}{\thestep.\arabic{substep}}

\newcounter{claim}[stepenv]
\newenvironment{claim}[1][]{\refstepcounter{claim}\par\medskip\noindent%
        \textit{Claim~\theclaim. #1} \itshape\rmfamily}{\medskip}

\newcommand{\cA}{\mathcal{A}}
\newcommand{\cB}{\mathcal{B}}

\newcommand{\bP}{\mathbb{P}}

\newcommand{\bR}{\mathbb{R}}

\newcommand{\vx}{\mathbf{x}}
\newcommand{\vy}{\mathbf{y}}
\newcommand{\vz}{\mathbf{z}}

\def\eps{{\epsilon}}

\newcommand{\defn}{\emph}

\newcommand{\expn}[1]{\mathrm{\mathbb{E}}\left[#1\right]}

\def\lcl{\left\lceil}
\def\rcl{\right\rceil}

\newcommand{\Set}[1]{\{#1\}}
\newcommand{\set}[2]{\{#1\,:\;#2\}}
\def\In{\subset}

\newcommand{\mc}{{\rm mc}}
\renewcommand{\sp}{{\rm sp}}

\def\COMMENT#1{}
\def\TASK#1{}
\let\TASK=\footnote             
\let\COMMENT=\footnote          

\begin{document}

\title{\vspace{-0.9cm}New results for MaxCut in $H$-free graphs}

\author{Stefan Glock \thanks{Institute for Theoretical Studies, ETH Z\"urich, Switzerland.
		Email: \href{mailto:dr.stefan.glock@gmail.com}{\nolinkurl{dr.stefan.glock@gmail.com}}.
		Research supported by Dr. Max R\"ossler, the Walter Haefner Foundation and the ETH Z\"urich Foundation.}
	\and Oliver Janzer \thanks{Department of Mathematics, ETH Z\"urich, Switzerland. Email: \href{mailto:oliver.janzer@math.ethz.ch} {\nolinkurl{oliver.janzer@math.ethz.ch}}.
	Research supported by an ETH Z\"urich Postdoctoral Fellowship 20-1 FEL-35.}
	\and Benny Sudakov \thanks{Department of Mathematics, ETH Z\"urich, Switzerland. Email:
	\href{mailto:benny.sudakov@gmail.com} {\nolinkurl{benny.sudakov@gmail.com}}.
	Research supported in part by SNSF grant 200021-175573.}
}

\date{}

\maketitle

\begin{abstract} 
The MaxCut problem asks for the size ${\rm mc}(G)$ of a largest cut in a graph~$G$. It is well known that ${\rm mc}(G)\ge m/2$ for any $m$-edge graph $G$, and the difference ${\rm mc}(G)-m/2$ is called the \emph{surplus} of~$G$. The study of the surplus of $H$-free graphs was initiated by Erd\H{o}s and Lov\'asz in the 70s, who in particular asked what happens for triangle-free graphs.
This was famously resolved by Alon, who showed that in the triangle-free case the surplus is $\Omega(m^{4/5})$, and found constructions matching this bound.
We prove several new results in this area.
\begin{enumerate}[label=\rm{(\roman*)}]
    \item We show that for every fixed odd $r\ge 3$, any $C_r$-free graph with $m$ edges has surplus $\Omega_r\big(m^{\frac{r+1}{r+2}}\big)$. This is tight, as is shown by a construction of pseudorandom $C_r$-free graphs due to Alon and Kahale. It improves previous results of several researchers, and complements a result of Alon, Krivelevich and Sudakov which is the same bound when $r$ is even.
    
    \item Generalizing the result of Alon, we allow the graph to have triangles, and show that if the number of triangles is a bit less than in a random graph with the same density, then the graph has large surplus. For regular graphs our bounds on the surplus are sharp.
    
    \item We prove that an $n$-vertex  graph with few copies of $K_r$ and average degree $d$ has surplus $\Omega_r(d^{r-1}/n^{r-3})$, which is tight when $d$ is close to $n$ provided that a conjectured dense pseudorandom $K_r$-free graph exists. This result is used to improve the best known lower bound (as a function of $m$) on the surplus of $K_r$-free graphs.
\end{enumerate}
Our proofs combine techniques from semidefinite programming,  probabilistic reasoning, as well as combinatorial and spectral arguments.
\end{abstract}

\section{Introduction}

\emph{MaxCut} is a central problem in discrete mathematics and theoretical computer science. Given a graph $G$, a \defn{cut} is a partition of the vertex set into two parts, and its size is the number of edges going across.
The aim is to determine the maximum size of a cut, denoted here by~$\mc(G)$. This problem has received a lot of attention in the last 50 years, both from an algorithmic perspective in theoretical computer science, where the aim is to approximate $\mc(G)$ well for a given graph, and from an extremal perspective in combinatorics, where we are mainly interested in good bounds on $\mc(G)$ in terms of the number of vertices and/or edges of~$G$.

A folklore observation is that every graph with $m$ edges has a cut of size at least~$m/2$. This can be easily seen using a probabilistic argument, or a greedy algorithm.
Therefore, it is a fundamental question by how much this trivial bound can be improved.
Since it was already demonstrated by Erd\H{o}s~\cite{erdos:67} in the 60s that the factor $1/2$ cannot be improved in general, even if we consider very restricted families such as graphs of large girth, the natural parameterization is to consider the so-called \defn{surplus} of a graph $G$, denoted $\sp(G)$, which is the difference $\mc(G)-\frac{m}{2}$ of the optimal cut size to the greedy bound.
By a classical result of Edwards \cite{edwards:73,edwards:75}, every graph $G$ with $m$ edges has surplus 
\begin{align}
\sp(G) \ge \frac{\sqrt{8m+1}-1}{8},\label{edwards}
\end{align}
and this is tight whenever $G$ is a complete graph with an odd number of vertices. 

Although the bound $\Omega(\sqrt{m})$ on the surplus is optimal in general, it can be significantly improved for graphs that are ``far'' from complete. One natural way to enforce this is to forbid the containment of a fixed subgraph. The study of MaxCut in $H$-free graphs was initiated by Erd\H{o}s and Lov\'asz (see~\cite{erdos:79}) in the 70s, and has received significant attention since then (e.g.~\cite{alon:96,ABKS:03,AKS:05,CKLMST:18,FHM:21,PT:94,shearer:92,ZH:18}). 
For a graph $H$, define $\sp(m,H)$ as the minimum surplus $\sp(G)=\mc(G)-m/2$ over all $H$-free graphs $G$ with $m$ edges.

It was shown in~\cite{ABKS:03} that for every fixed graph $H$, there exist constants $\eps=\eps(H)>0$ and $c=c(H)>0$ such that $\sp(m,H)\ge c m^{1/2+\eps}$ for all~$m$. This demonstrates that the surplus of $H$-free graphs is significantly larger than the bound~\eqref{edwards} in Edwards' theorem. Perhaps the main conjecture in the area, due to Alon, Bollob\'{a}s, Krivelevich and Sudakov~\cite{ABKS:03}, is that the constant $1/2$ in the aforementioned result can be replaced with~$3/4$, which would be optimal as can be seen by considering a random graph with edge probability $p=n^{-\delta}$ for some arbitrarily small $\delta>0$ and $H=K_r$ sufficiently large.
This conjecture is still wide open. In fact, we do not even know whether there exists a constant $\alpha>1/2$ such that uniformly for all $H$, we have $\sp(m,H)\ge c_H m^{\alpha}$, where $c_H>0$ may depend on~$H$.

An even more ambitious problem, first explicitly posed in \cite{AKS:05}, is to determine the asymptotic growth rate of $\sp(m,H)$ for every fixed graph~$H$. However, results of the type $\sp(m,H)=\Theta(m^{q_H})$ for some constant $q_H$ depending only on $H$, are still very rare. 
The case when $H$ is a triangle has received particularly much attention. Erd\H{o}s and Lov\'asz (see~\cite{erdos:79}) first showed that 
$\sp(m,K_3)=  \Omega\left(m^{2/3}\left(\log m/\log \log m\right)^{1/3}\right)$. 
They also noted that it seems unclear what the correct exponent of $m$ should be, even in this elementary case, which testifies about the difficulty of the problem.
The logarithmic factor was later improved by Poljak and Tuza~\cite{PT:94}. Shearer~\cite{shearer:92} proved an important bound on the surplus of any triangle-free graph $G$ showing that
\begin{align}
\sp(G)=\Omega\left(\sum_{v\in V(G)} \sqrt{d_G(v)}\right).\label{shearer}
\end{align}
This bound is tight in general and implies that $\sp(m,K_3)= \Omega(m^{3/4})$.
Finally, Alon~\cite{alon:96} proved that
$\sp(m,K_3)=\Theta(m^{4/5})$. His result is exemplary for many of the challenges and developed methods in the area.

Our contribution in this paper is threefold. First, we determine $q_H$ for all odd cycles, which improves earlier results of several researchers and adds to the lacunary list of graphs for which $q_H$ is known.
Secondly, we extend the result of Alon in the sense that we prove optimal bounds on the surplus of general graphs in terms of the number of triangles they contain.
Thirdly, we study the surplus of graphs with few cliques of size $r$ and improve the currently best bounds for $K_r$-free graphs.

\subsection{Cycles of odd length}\label{subsec:odd cycles}

The study of MaxCut in graphs without short cycles goes back to the work of Erd\H{o}s and Lov\'asz (see~\cite{erdos:79}).
Motivated by one of their conjectures, Alon, Bollob\'{a}s, Krivelevich and Sudakov~\cite{ABKS:03} showed that $m$-edge graphs not containing any cycle of length at most $r$ have surplus $\Omega_r(m^{(r+1)/(r+2)})$. 
Later, Alon, Krivelevich and Sudakov~\cite{AKS:05} proved that 
\begin{align}
\sp(m,C_r)=\Omega_r\left(m^{(r+1)/(r+2)}\right)\label{cycle surplus}
\end{align}
for all even~$r$. They also conjectured that this bound is tight. It is well known (see e.g.~\cite{alon:96,DP:93,MP:90})
that if a graph $G$ has smallest eigenvalue $\lambda_{\rm{min}}$, then
\begin{align}
\sp(G)\le |\lambda_{\rm{min}}||G|/4.\label{eigenvalue bound}
\end{align}
In fact, essentially all known extremal examples for the MaxCut problem of $H$-free graphs come from $(n,d,\lambda)$-graphs (i.e., $n$-vertex $d$-regular graphs where each nontrivial eigenvalue has absolute value at most $\lambda$).
For instance, the well-known Erd\H{o}s--R\'enyi graphs~\cite{ER:62} are $C_4$-free $(n,d,\lambda)$-graphs with $d=\Theta(\sqrt{n})$ and $\lambda=O(\sqrt{d})$, which arise as polarity graphs of finite projective planes. By~\eqref{eigenvalue bound}, these graphs have surplus at most $O(\lambda n)=O((nd)^{5/6})$, which shows the optimality of the bound~\eqref{cycle surplus} in the case $r=4$.
Similar constructions work in the cases $r\in \Set{6,10}$ (see~\cite{AKS:05} for more details).
However, showing the tightness of~\eqref{cycle surplus} in general is probably very hard since it seems to go hand in hand with the construction of extremal examples for the Tur\'an problem of even cycles.

Another obvious problem left open by the work of Alon, Krivelevich and Sudakov~\cite{AKS:05} is to establish~\eqref{cycle surplus} for odd cycles too.
One advantage here is that the conjectured extremal examples are already known. Indeed, the construction of triangle-free dense pseudorandom graphs due to Alon~\cite{alon:94} generalizes to longer odd cycles, as was observed by Alon and Kahale~\cite{AK:98} (see also~\cite{KS:06}). This yields, for every fixed odd $r$, $C_r$-free $(n,d,\lambda)$-graphs with $d=\Theta(n^{2/r})$ and $\lambda=O(\sqrt{d})$. By~\eqref{eigenvalue bound}, these graphs have surplus at most $O(\lambda n)=O((nd)^{(r+1)/(r+2)})$, which shows that~\eqref{cycle surplus} would be optimal for all odd~$r$.
Regarding this problem, Zeng and Hou~\cite{ZH:18} showed that $\sp(m,C_r)\ge m^{(r+1)/(r+3)+o(1)}$ for all odd $r$, and very recently Fox, Himwich and Mani~\cite{FHM:21} improved the surplus to $\Omega_r(m^{(r+5)/(r+7)})$.
We settle this problem completely by proving the following tight result.

\begin{theorem}\label{thm:odd cycles}
For odd $r\ge 3$ there is a constant $\alpha_r>0$ such that any $C_r$-free graph with $m$ edges has a cut of size at least $m/2+\alpha_r\cdot m^{(r+1)/(r+2)}$. This is tight up to the value of $\alpha_r$.
\end{theorem}

\subsection{Few triangles} \label{subsec:intro triangles}
Shearer's bound~\eqref{shearer} implies that any triangle-free $d$-degenerate graph with $m$ edges has surplus
$\Omega(m/\sqrt{d})$. In his proof of $\sp(m,K_3)=\Omega(m^{4/5})$, Alon \cite{alon:96} combined this estimate with an additional argument which shows that a triangle-free graph with average degree $d$ has surplus $\Omega(d^2)$. 
This supersedes Shearer's result when $d$ is larger than roughly $m^{2/5}$. It is natural to ask what happens with both bounds if we do not insist that the graph is triangle-free, but only has few triangles.

Alon, Krivelevich and Sudakov~\cite{AKS:05} proved an extension of Shearer's bound~\eqref{shearer} to graphs with few triangles. They showed that if every vertex $v\in V(G)$ is contained in $o(d_G(v)^{3/2})$ triangles, then (\ref{shearer}) remains true. Their motivation was to use this as a tool to study the surplus of $H$-free graphs when $H$ is an even cycle or a small complete bipartite graph, since then neighbourhoods are sparse and hence every vertex is contained in few triangles. 
Recently, Carlson, Kolla, Li, Mani, Sudakov and Trevisan~\cite{CKLMST:18} showed that the local assumption can be relaxed to a global condition, namely, any $d$-degenerate graph with $m$ edges and $O(m \sqrt{d})$ triangles has surplus $\Omega(m/\sqrt{d})$. Their proof is based on a probabilistic rounding of the solution of the semidefinite programming relaxation of MaxCut. This is one of our main tools too, and we will discuss it in more detail in Section~\ref{sec:illustration}.

While the above extends Shearer's bound to graphs with few triangles, it is also natural to ask how Alon's second bound, the surplus $\Omega(d^2)$, is affected when we allow triangles. 
We answer this question, proving that if $G$ has less triangles than a random graph of the same average degree, it has large surplus.

\begin{theorem}\label{thm:triangle simple}
	Let $G$ be a graph with average degree $d$ and less than $(1-\eps)d^3/6$ triangles, for some $\eps>0$. Then $G$ has surplus $\Omega_{\eps}(d^2)$.
\end{theorem}

In the case of regular graphs, we obtain even a much sharper result, which expresses the surplus as a function of the order, degree, and number of triangles of the graph.

\begin{theorem} \label{thm:triangledependence}
    Let $G$ be a $d$-regular graph with $n$ vertices and $d^3/6+s$ triangles, where $d\leq n/2$.
    \begin{itemize}
        \item If $s<-nd^{3/2}$, then $G$ has surplus $\Omega(\frac{|s|}{d})$.
        \item If $-nd^{3/2}\leq s \leq nd^{3/2}$, then $G$ has surplus $\Omega(nd^{1/2})$.
        \item If $s>nd^{3/2}$, then $G$ has surplus $\Omega(\frac{n^2d^2}{s})$.
    \end{itemize}
\end{theorem}

The bounds given here are tight for all strongly regular graphs that attain the respective parameters (see Theorem~\ref{thm:strongly regular surplus} for the exact statement). Also, after a simple modification of the proof of Theorem~\ref{thm:triangledependence}, the condition $d\leq n/2$ can be replaced by $d\leq (1-\eps)n$ for any positive constant $\eps$. We remark that the conclusion of Theorem~\ref{thm:triangledependence} is no longer true if we only assume that the average degree of $G$ is $d$, rather than that it is $d$-regular. To see this, take for example $G$ to be the graph which is the union of a random graph with $n^{4/5}$ vertices and density $n^{-2/5}$ and an empty graph on $n-n^{4/5}$ vertices. It is not hard to see that with high probability, $G$ has average degree $d=\Theta(n^{1/5})$, the number of triangles in $G$ is $t(G)=\Theta(n^{6/5})$, but the surplus is only $\Theta(n)$ rather than $\Omega(nd^{1/2})$.

\subsection{Cliques}

Recall the conjecture of Alon, Bollob\'{a}s, Krivelevich and Sudakov~\cite{ABKS:03} that for every graph $H$, there exist constants $\eps=\eps(H)>0$ and $c=c(H)>0$ such that $\sp(m,H)\ge c m^{3/4+\eps}$ for all~$m$. Clearly, it suffices to prove this conjecture for cliques.
Carlson et al.~\cite{CKLMST:18} speculated that the following analogue of Shearer's bound should hold for $K_r$-free graphs, and proved that this would imply the conjecture of Alon, Bollob\'{a}s, Krivelevich and Sudakov.

\begin{conj}[\cite{CKLMST:18}]\label{conj:degeneracy}
For every fixed $r\ge 3$, any $K_r$-free $d$-degenerate graph with $m$ edges has surplus $\Omega_r(m/\sqrt{d})$.
\end{conj}

They also proved a weakening of this conjecture, namely the lower bound $\Omega_r(m/d^{1-1/(2r-4)})$.
We improve this bound slightly. It would be very interesting to obtain a version of this result where the exponent of $d$ is bounded by a constant smaller than $1$ uniformly for all~$r$.

\begin{theorem}\label{thm:clique degeneracy}
For every fixed $r\ge 3$, any $K_r$-free $d$-degenerate graph with $m$ edges has surplus $\Omega_r(m/d^{1-1/(r-1)})$.
\end{theorem}

As discussed in Section~\ref{subsec:odd cycles}, essentially all known extremal constructions for MaxCut in $H$-free graphs are dense pseudorandom graphs.
Motivated by this observation, one can even make a more precise guess for $\sp(m,K_r)$.
However, the construction of dense pseudorandom $K_r$-free graphs is a major open problem in its own right. Sudakov, Szab\'{o} and Vu~\cite{SSV:05} raised the question whether, for fixed $r\ge 3$, there exist $K_r$-free $(n,d,\lambda)$-graphs with $d=\Omega(n^{1-1/(2r-3)})$ and $\lambda=O(\sqrt{d})$. It has been proved that this would have striking consequences for off-diagonal Ramsey numbers~\cite{MV:19}.
Speculating that such graphs do exist, this would via~\eqref{eigenvalue bound} imply that 
\begin{align}
\sp(m,K_r)=O\left(m^{1-\frac{1}{4+1/(r-2)}}\right).\label{clique surplus upper}
\end{align}
Perhaps this gives the correct answer.
Note that Conjecture~\ref{conj:degeneracy} would provide the matching lower bound as long as $d=O(n^{1-1/(2r-3)})$, that is, in the ``sparse'' case, when the density is smaller than that of the speculative extremal examples.
We prove a result that complements this in the ``dense'' regime, analogous to Alon's $\Omega(d^2)$ surplus for triangle-free graphs.

\begin{theorem}\label{thm:clique free simple}
For every fixed $r\ge 3$, any $K_r$-free $n$-vertex graph with average degree $d$ has surplus $\Omega_r\left(d^{r-1}/n^{r-3}\right)$.
\end{theorem}

For $d=\Omega(n^{1-1/(2r-3)})$, the implied surplus is at least $\Omega_r\left((nd)^{1-\frac{1}{4+1/(r-2)}}\right)$. Hence,
if Conjecture~\ref{conj:degeneracy} is true, then combining it with Theorem~\ref{thm:clique free simple} similar as in Alon's result, we could deduce that indeed $\sp(m,K_r)=\Omega_r\left(m^{1-\frac{1}{4+1/(r-2)}}\right)$, matching the speculative upper bound in~\eqref{clique surplus upper}. Using Theorem~\ref{thm:clique degeneracy} instead of Conjecture~\ref{conj:degeneracy}, we still improve the current best known lower bound $\sp(m,K_r)=\Omega_r(m^{\frac{1}{2}+\frac{1}{4r-8}})$, due to Carlson et al.~\cite{CKLMST:18}.

\begin{theorem} \label{thm:clique free in terms of m}
    For every $r\geq 3$, any $K_r$-free graph with $m$ edges has surplus $\Omega_r(m^{\frac{1}{2}+\frac{3}{4r-2}})$.
\end{theorem}

In fact, we will establish Theorem~\ref{thm:clique free simple} in the stronger form that it is sufficient that the number of $r$-cliques is smaller than in a random graph of the same density. This generalizes Theorem~\ref{thm:triangle simple}.

\begin{theorem}\label{thm:clique optimal}
    For any integer $r\geq 3$ and any $\eps>0$, there exists a positive constant $c=c(r,\eps)$ such that the following holds. If $G$ is a graph on $n$ vertices with average degree $d$ which has at most $(1-\eps)\frac{n^r}{r!}(d/n)^{\binom{r}{2}}$ copies of $K_r$, then $G$ has surplus at least $cd^{r-1}/n^{r-3}$.
\end{theorem}

Finally, we remark that Theorem~\ref{thm:clique free simple} generalizes a result of Nikiforov~\cite{nikiforov:06}. He proved that the smallest eigenvalue of a $K_r$-free $n$-vertex graph with average degree $d$ has absolute value $\Omega(d^{r-1}/n^{r-2})$. This also follows from Theorem~\ref{thm:clique free simple} and~\eqref{eigenvalue bound}. However, a lower bound on MaxCut is stronger than a bound on $\lambda_{\rm{min}}$ since the converse of~\eqref{eigenvalue bound} is not true.
Theorem~\ref{thm:clique optimal} shows that the conclusion of Nikiforov's theorem remains true even if we only require that there are ``few'' copies of~$K_r$.

\vspace{0.3cm}
\noindent
{\bf Notation:} We use standard notation. All graphs considered are finite, simple and undirected. We write $|G|$ for the order of $G$, $e(G)$ for the number of edges, and $t(G)$ for the number of triangles in~$G$. The minimum and maximum degree of $G$ are denoted by $\delta(G)$ and $\Delta(G)$, respectively. For a set $S\subset V(G)$, we write $e_G(S)$ for the number of edges in $G[S]$, and for disjoint sets $S,T\subset V(G)$, we write $e_G(S,T)$ for the number of edges in $G$ with one endpoint in $S$ and one in $T$. $N_G(v)$ and $d_G(v)$ stand for the neighbourhood and degree of a vertex $v$. The number of common neighbours of vertices $v_1,\dots,v_r$ is denoted $d_G(v_1,\dots,v_r)$. 
We occasionally omit subscripts if they are clear from the context.
We let $P_n$ and $C_n$ denote the path and cycle, respectively, with $n$ edges.

When using Landau symbols, $\Omega_r(\cdot),\Theta_r(\cdot),O_r(\cdot)$ mean that the implicit constants may depend on~$r$.
We denote by $\|\cdot\|$ the Euclidean norm.
As customary, we tacitly treat large numbers like integers whenever this has no effect on the argument.

The rest of this paper is organized as follows. In Section~\ref{sec:illustration}, we introduce some of the main tools utilised in our proofs, and use them to prove a few of our less technical results. In particular, we first prove Theorem~\ref{thm:triangledependence} in the special case where $G$ is strongly regular. Also, we present the proof of Theorem~\ref{thm:triangle simple}. We conclude the section with the proof of Theorem~\ref{thm:odd cycles} for the simplest new case $C_5$ (the case $r=3$ is already covered by Alon's theorem) under the assumption that $G$ is regular.
In Section~\ref{sec:odd cycles}, we prove Theorem~\ref{thm:odd cycles} in the general case.
We give the full proof of Theorem~\ref{thm:triangledependence} in Section~\ref{sec:few triangles}.
Finally, we prove Theorems~\ref{thm:clique degeneracy},~\ref{thm:clique free in terms of m} and~\ref{thm:clique optimal} in Section~\ref{sec:cliques}.

\section{Main ideas and first applications} \label{sec:illustration}

\subsection{Preliminaries}

We start with a well known observation that we will repeatedly use in our proofs. In order to show that a graph has large surplus, it suffices to find many disjoint induced subgraphs with relatively large surplus.

\begin{lemma} \label{lemma:surplus additive}
    Let $G$ be a graph and let $S_1,S_2,\dots,S_k$ be pairwise disjoint subsets of $V(G)$. Then $\sp(G)\geq \sum_{i=1}^k \sp(G[S_i])$.
\end{lemma}

Variants of this have been used in many preceding papers (see e.g.~\cite[Lemma~3.2]{AKS:05}).
The proof is similar to the standard $m/2$ bound for $\mc(G)$. Note that by adding singletons, one can assume that the $S_i$'s partition $V(G)$. Take an optimal cut of each of the induced subgraphs, and form a cut of $G$ by putting the sets of the smaller cuts randomly on either side. The result then follows by inspecting the expected size of the obtained cut. We omit the details.

Note that in the simplest case, the lemma says that the surplus of a graph is at least as large as the surplus of any of its induced subgraphs. We will use this fact naturally in our proofs without necessarily referring to the lemma.

\subsection{Semidefinite programming} \label{subsec:SDP}

Semidefinite programming in the algorithmic context of MaxCut was first used by Goemans and Williamson~\cite{GW:95}. The method was exploited by Carlson et al.~\cite{CKLMST:18} to give lower bounds on the MaxCut for $H$-free graphs. The approach can be summarised as follows. Given a graph $G$, we try to assign a unit vector $\vx^v\in \mathbb{R}^N$ to each vertex $v\in V(G)$ in a way that the inner products $\langle \vx^u,\vx^v \rangle$ are negative on average over all adjacent pairs of vertices $u,v$. Then we take a random unit vector $\vz\in \mathbb{R}^N$ and define $A=\{v\in V(G): \langle \vx^v,\vz \rangle\geq 0\}$, $B=V(G)\setminus A$. Consider the random cut $(A,B)$. If vertices $u$ and $v$ have $\langle \vx^u,\vx^v\rangle<0$, then with probability more than $1/2$, $u$ and $v$ will end up in different parts of the cut. Hence, intuitively, we expect our cut to contain more than half of the edges. This is made precise by the following lemma, essentially due to Goemans and Williamson \cite{GW:95}.

\begin{lemma} \label{lemma:general SDP}
    Let $G$ be a graph and let $N$ be a positive integer. Then, for any set of non-zero vectors $\{\vx^v:v\in V(G)\}\subset \mathbb{R}^N$, we have $\sp(G)\geq -\frac{1}{\pi}\sum_{uv\in E(G)} \arcsin\left(\frac{\langle \vx^u,\vx^v \rangle}{\|\vx^u\|\|\vx^v\|}\right)$.
\end{lemma}

\begin{proof}
Let $\vz$ be a uniformly random unit vector in $\mathbb{R}^N$, let $A=\{v\in V(G): \langle \vx^v,\vz \rangle\geq 0\}$ and let $B=V(G)\setminus A$. The probability that $uv$ belongs to the cut $(A,B)$ is $\frac{\alpha}{\pi}$, where $\alpha=\arccos\left(\frac{\langle \vx^u,\vx^v \rangle}{\|\vx^u\|\|\vx^v\|}\right)$ is the angle between $\vx^u$ and $\vx^v$. Hence,
$$\bP(uv \text{ belongs to the cut})=\frac{1}{\pi}\arccos\left(\frac{\langle \vx^u,\vx^v \rangle}{\|\vx^u\|\|\vx^v\|}\right)=\frac{1}{2}-\frac{1}{\pi}\arcsin\left(\frac{\langle \vx^u,\vx^v \rangle}{\|\vx^u\|\|\vx^v\|}\right).$$
Summing the equality over all edges $uv\in E(G)$ shows that the expected number of edges in the cut $(A,B)$ is $\frac{e(G)}{2}-\frac{1}{\pi}\sum_{uv\in E(G)} \arcsin\left(\frac{\langle \vx^u, \vx^v \rangle}{\|\vx^u\|\|\vx^v\|}\right)$, from which the result follows.
\end{proof}

	In some cases, we can use a more convenient form of the above, which is implicit in the work of Carlson et al.~\cite{CKLMST:18}.
	
	\begin{cor} \label{cor:general SDP}
	Let $G$ be a graph and let $N$ be a positive integer. Let $\{\vx^v:v\in V(G)\}\subset \mathbb{R}^N$ be vectors with $\|\vx^v\|=O(1)$, and suppose that for every edge $uv\in E(G)$, we have $\langle \vx^u,\vx^v \rangle \le -a_{uv}+b_{uv}$ for some $a_{uv},b_{uv}\ge 0$.
    Then
	$$\sp(G)\geq \Omega\left(\sum_{uv\in E(G)}a_{uv}\right) -  O\left(\sum_{uv\in E(G)}b_{uv}\right).$$
	\end{cor}
	
	\begin{proof}
	We add extra coordinates, one for each vertex $v\in V(G)$. The vector $\vy^v$ is obtained from $\vx^v$ by setting the new coordinate for $v$ to $1$ and all other new coordinates to~$0$.
	Hence, we have $\|\vy^v\|^2=\|\vx^v\|^2+1 = \Theta(1)$ for all $v\in V(G)$, and 
	$\langle \vy^u,\vy^v \rangle = \langle \vx^u,\vx^v \rangle \le -a_{uv}+b_{uv}$ for all $uv\in E(G)$.
	
	For any $x\in[-1,1]$ with $x\le b-a$ for some $a,b\ge 0$, we have $\arcsin(x)\le \frac{\pi}{2}b-a$. This is because when $x<0$, then $\arcsin(x)\le x \le b-a \le \frac{\pi}{2}b-a$, and if $x\ge 0$, then $\arcsin(x)\le \frac{\pi}{2}x\le \frac{\pi}{2}(b-a) \le \frac{\pi}{2}b-a$.
	Hence, for every edge $uv\in E(G)$, we have 
	$$\arcsin\left(\frac{\langle \vy^u,\vy^v \rangle}{\|\vy^u\|\|\vy^v\|}\right) \le \frac{\pi}{2}\frac{b_{uv}}{\|\vy^u\|\|\vy^v\|} - \frac{a_{uv}}{\|\vy^u\|\|\vy^v\|} = -\Omega(a_{uv}) + O(b_{uv}).$$ Summing over all edges and using Lemma~\ref{lemma:general SDP} gives the result.
	\end{proof}

\subsection{Surplus in strongly regular graphs}

As a first illustration of the use of semidefinite programming, we determine the order of magnitude of the surplus in any strongly regular graph. Recall that a graph is called \defn{strongly regular} if in addition to being regular, there exist $\eta,\mu$ such that any two adjacent vertices have exactly $\eta$ common neighbours and any two distinct non-adjacent vertices have exactly $\mu$ common neighbours.

In Section~\ref{sec:few triangles}, we will use a more involved argument to generalize the lower bound to every regular graph, proving Theorem~\ref{thm:triangledependence}.

\begin{theorem} \label{thm:strongly regular surplus}
    Let $G$ be a strongly regular graph with $n$ vertices, degree $d$ and $d^3/6+s$ triangles, where $d\leq 0.99n$.
    \begin{itemize}
        \item If $s<-nd^{3/2}$, then $\sp(G)=\Theta(\frac{|s|}{d})$.
        \item If $-nd^{3/2}\leq s \leq nd^{3/2}$, then  $\sp(G)=\Theta(nd^{1/2})$.
        \item If $s>nd^{3/2}$, then $\sp(G)=\Theta(\frac{n^2d^2}{s})$.
    \end{itemize}
\end{theorem}

\begin{proof}
Let us start with the lower bounds. Let $a=\frac{\sqrt{d}}{n-d}$ and let $0<\gamma\leq 1$. For every $v\in V(G)$, we define a vector $\vx^v$ whose coordinates are labelled by $V(G)$, i.e. $\vx^v\in \bR^{V(G)}$. Take
$$\vx^v_u=
\begin{cases}
1+\gamma a \text{ if } u=v, \\
-\gamma \frac{1}{\sqrt{d}} \text{ if } u\in N(v), \\
\gamma a \text{ otherwise.}
\end{cases}
$$
Then for any $uv\in E(G)$,
\begin{align*}
    \langle \vx^u,\vx^v\rangle
    &=-2(1+\gamma a)\gamma\frac{1}{\sqrt{d}}+|N(u)\cap N(v)|\gamma^2\frac{1}{d}+|V(G)\setminus (N(u)\cup N(v))|\gamma^2a^2 \\
    &+\bigg(|N(u)\setminus (N(v)\cup \{v\})|+|N(v)\setminus (N(u)\cup \{u\})|\bigg)\left(-\gamma^2\frac{a}{\sqrt{d}}\right).
\end{align*}
Observe that $|N(u)\cap N(v)|=d(u,v)$, $|V(G)\setminus (N(u)\cup N(v))|=n-2d+d(u,v)$ and $|N(u)\setminus (N(v)\cup \{v\})|=|N(v)\setminus (N(u)\cup \{u\})|=d-d(u,v)-1$, so
$$\langle \vx^u,\vx^v\rangle=-2\gamma\frac{1}{\sqrt{d}}+\gamma^2\left(\left(\frac{1}{d}+2\frac{a}{\sqrt{d}}+a^2\right)d(u,v)+(n-2d)a^2-2d\frac{a}{\sqrt{d}}\right).$$

Let $\delta(u,v)=d(u,v)-\frac{d^2}{n}$. Note that
$$\left(\frac{1}{d}+2\frac{a}{\sqrt{d}}+a^2\right)\frac{d^2}{n}+(n-2d)a^2-2d\frac{a}{\sqrt{d}}=\frac{1}{n}(d-2a\sqrt{d}(n-d)+a^2(n-d)^2)=\frac{1}{n}(d-2d+d)=0,$$
so
\begin{equation}
    \langle \vx^u,\vx^v\rangle=-2\gamma \frac{1}{\sqrt{d}}+\gamma^2 \left(\frac{1}{d}+2\frac{a}{\sqrt{d}}+a^2\right)\delta(u,v). \label{eqn:innerproduct}
\end{equation}
Note that so far we have not needed the assumption that $G$ is strongly regular, only that it is regular. We will in fact use (\ref{eqn:innerproduct}) also in the proof of Theorem \ref{thm:triangledependence}.

Since $G$ is strongly regular, every pair of adjacent vertices has the same number $\eta$ of common neighbours. Note that $e(G)\eta=3t(G)$, so \begin{equation}
    \eta=\frac{3t(G)}{e(G)}=\frac{3(d^3/6+s)}{nd/2}=\frac{d^2}{n}+\frac{6s}{nd}, \label{eqn:eta}
\end{equation}
hence $\delta(u,v)=\frac{6s}{nd}$ for every $uv\in E(G)$. Thus,
\begin{equation*}
    \langle \vx^u,\vx^v\rangle=-2\gamma \frac{1}{\sqrt{d}}+\gamma^2 \left(\frac{1}{d}+2\frac{a}{\sqrt{d}}+a^2\right)\frac{6s}{nd}.
\end{equation*}
Observe that since $d\leq 0.99n$, we have $a\leq \frac{100}{\sqrt{d}}$, so $ \frac{1}{d}\leq \frac{1}{d}+2\frac{a}{\sqrt{d}}+a^2\leq \frac{10^5}{d}$.

If $s<-nd^{3/2}$, then take $\gamma=1$ to get
$$\langle \vx^u,\vx^v\rangle\leq \frac{1}{d}\cdot\frac{6s}{nd}=\frac{6s}{nd^2}.$$

If $-nd^{3/2}\leq s\leq nd^{3/2}$, take $\gamma=1/10^6$ and note that $\gamma^2 \left(\frac{1}{d}+2\frac{a}{\sqrt{d}}+a^2\right)\frac{6s}{nd}\leq \gamma\frac{1}{\sqrt{d}}$, so
$$\langle \vx^u,\vx^v\rangle\leq -\gamma\frac{1}{\sqrt{d}}=-\frac{1}{10^6\sqrt{d}}.$$

Finally, if $s>nd^{3/2}$, take $\gamma=\frac{nd^{3/2}}{10^6s}$ and note that $\gamma^2 \left(\frac{1}{d}+2\frac{a}{\sqrt{d}}+a^2\right)\frac{6s}{nd}\leq \gamma\frac{1}{\sqrt{d}}$, so
$$\langle \vx^u,\vx^v\rangle\leq -\gamma\frac{1}{\sqrt{d}}=-\frac{nd}{10^6s}.$$

In all three cases, $\langle \vx^u,\vx^v\rangle<0$, and moreover
it is easy to see that $\|\vx^v\|=O(1)$ for every $v\in V(G)$, so by Corollary~\ref{cor:general SDP}, we have
$$\sp(G)\geq \Omega\left(\sum_{uv\in E(G)} -\langle \vx^u,\vx^v\rangle\right).$$
Summing the respective upper bounds for $\langle \vx^u,\vx^v\rangle$ in the three cases, we obtain the desired lower bounds for $\sp(G)$.

It remains to prove the upper bounds for $\sp(G)$. For this, we will use the eigenvalue bound (\ref{eigenvalue bound}). It is well known (see \cite{KS:06}, for example) that the smallest eigenvalue of a strongly regular graph is
\begin{equation*}
    \lambda_{\min} = \frac{1}{2}\left(\eta-\mu - \sqrt{(\eta-\mu)^2 + 4(d-\mu)} \right),
\end{equation*}
where $\eta$ is the number of common neighbours of adjacent pairs and $\mu$ is the number of common neighbours of non-adjacent pairs. It is a straightforward (but somewhat tedious) exercise to check that this, using \eqref{eqn:eta}, implies the desired bounds. For the details, see Lemma~\ref{lemma:eigenvalue computation} in the appendix.
\end{proof}

\subsection{Finding good almost regular induced subgraphs}

In this subsection we establish an important lemma that allows us to reduce many of our statements to the special case where $G$ is almost regular. For example, in the next subsection we already 
use it to study the surplus of graphs with few triangles. We believe that this lemma will be a useful tool for future works on this subject as well.

\begin{lemma} \label{lemma:regularize basic}
    Let $\alpha$ and $\beta$ be real numbers such that $\beta>0$, $\alpha<\beta$ and $\alpha+\beta\leq 2$. Then there exist positive constants $c_1$, $c_2$ and $C$ (depending on $\alpha,\beta$), such that the following holds.
    
    Let $G$ be a graph on $n$ vertices with average degree $d$. Then either $G$ has surplus at least $c_1n^{\alpha}d^{\beta}$ or $G$ has an induced subgraph $\tilde{G}$ with $\tilde{n}$ vertices and average degree $\tilde{d}$ where $\tilde{n}^{\alpha}\tilde{d}^{\beta}\geq c_2n^{\alpha}d^{\beta}$, $\delta(\tilde{G})\geq \tilde{d}/C$ and $\Delta(\tilde{G})\leq C\tilde{d}$.
\end{lemma}

A typical use of this lemma will be as follows. Suppose that we want to prove that an $n$-vertex $H$-free graph with average degree $d$ has surplus $\Omega(n^{\alpha}d^{\beta})$. Applying Lemma~\ref{lemma:regularize basic}, we can pass to the induced subgraph $\tilde{G}$ if necessary, so it suffices to prove the bound $\Omega(n^{\alpha}d^{\beta})$ under the assumption that $\delta(G)\geq d/C$ and $\Delta(G)\leq Cd$.

For a few of our results, we will need the following somewhat stronger version of Lemma~\ref{lemma:regularize basic}. Similarly to the previous lemma, one could also impose an additional minimum degree condition $\delta(\tilde{G})\geq \tilde{d}/C$, but we will not need that.

\begin{lemma} \label{lemma:regularize tight}
    Let $\eps$, $\alpha$ and $\beta$ be real numbers such that $\eps,\beta>0$, $\alpha<\beta$ and $\alpha+\beta\leq 2$. Then there exist positive constants $c$ and $C$ (depending on $\eps,\alpha,\beta$), such that the following holds.
    
    Let $G$ be a graph on $n$ vertices with average degree $d$. Then either $G$ has surplus at least $cn^{\alpha}d^{\beta}$ or $G$ has an induced subgraph $\tilde{G}$ with $\tilde{n}$ vertices and average degree $\tilde{d}$ where $\tilde{n}^{\alpha}\tilde{d}^{\beta}\geq (1-\eps)n^{\alpha}d^{\beta}$ and $\Delta(\tilde{G})\leq C\tilde{d}$.
\end{lemma}

\begin{proof}
Choose $\theta$ such that $(1-\theta)^{\beta}=1-\eps$. Let $c=\frac{\theta^2}{320}$, let $C_0$ be chosen so that $(\theta^2/20)^{\beta}C_0^{\beta-\alpha}= 1$ and let $C=C_0/(1-\theta)$.

\begin{claim}
    If $G$ is a graph with $n$ vertices, $m$ edges and average degree $d$, then at least one of the following must hold.
    \begin{enumerate}[label=\rm{(\roman*)}]
        \item $G$ has an induced subgraph $G'$ with $n'<n$ vertices and average degree $d'$, where $(n')^{\alpha}(d')^{\beta}\geq n^{\alpha}d^{\beta}$.
        \item $G$ has surplus at least $\frac{\theta^2}{160}m$.
        \item $G$ has an induced subgraph $\tilde{G}$ with $\tilde{n}$ vertices and average degree $\tilde{d}$, where $\tilde{n}^{\alpha}\tilde{d}^{\beta}\geq (1-\eps)n^{\alpha}d^{\beta}$ and $\Delta(\tilde{G})\leq C\tilde{d}$.
    \end{enumerate}
\end{claim}

Before proving the claim, let us see how the lemma follows. Given a graph $G$, apply the claim repeatedly for the obtained induced subgraphs $G'$, $(G')'$, etc. as long as (i) holds. Eventually, (ii) or (iii) must hold, so $G$ has an induced subgraph $F$ with $N$ vertices and average degree $D$ such that $N^{\alpha}D^{\beta}\geq n^{\alpha}d^{\beta}$, and either $F$ has surplus at least $\frac{\theta^2}{320}ND$, or $F$ has an induced subgraph $\tilde{G}$ with $\tilde{n}$ vertices and average degree $\tilde{d}$, where $\tilde{n}^{\alpha}\tilde{d}^{\beta}\geq (1-\eps)N^{\alpha}D^{\beta}$ and $\Delta(\tilde{G})\leq C\tilde{d}$. In the latter case, we are done. In the former case, note that since $\alpha+\beta\leq 2$, $\alpha<\beta$ and $N>D$, we have $ND\geq (ND)^{\frac{1}{2}(\alpha+\beta)}\geq N^{\alpha}D^{\beta}$, so $F$ (and, consequently, $G$) has surplus at least $\frac{\theta^2}{320}N^{\alpha}D^{\beta}\geq \frac{\theta^2}{320}n^{\alpha}d^{\beta}=cn^{\alpha}d^{\beta}$. Thus, it suffices to prove the claim.

\claimproof
Let $S$ be the set of vertices in $G$ with degree at most $C_0d$ and let $T=V(G)\setminus S$. We consider the following three cases (observe that at least one of these must occur).

\medskip

\textbf{Case 1.} $e_G(T)\geq \frac{\theta^2}{20}m$.

Let $G'=G[T]$ and let $n'$ and $d'$ be the number of vertices and average degree of $G'$, respectively. Then $n'd'=2e(G')\geq \theta^2 m/10=\theta^2 nd/20$, so $\frac{d'}{d}\geq \frac{\theta^2}{20}\frac{n}{n'}$ and hence
$$\frac{(n')^\alpha (d')^{\beta}}{n^{\alpha} d^{\beta}}\geq \left(\frac{n'}{n}\right)^{\alpha} \left(\frac{\theta^2}{20}\frac{n}{n'}\right)^{\beta}=(\theta^2/20)^{\beta}\left(\frac{n}{n'}\right)^{\beta-\alpha}.$$
Note that $G$ has average degree $d$ but every vertex in $T$ has degree at least $C_0d$ in $G$, so $n'=|T|\leq n/C_0$. Hence,
$$(\theta^2/20)^{\beta}\left(\frac{n}{n'}\right)^{\beta-\alpha}\geq (\theta^2/20)^{\beta}C_0^{\beta-\alpha}=1.$$
Thus, $G'$ satisfies (i).

\medskip

\textbf{Case 2.} $e_G(T)< \frac{\theta^2}{20}m$, but $e_G(S,T)\geq \frac{\theta}{2}m$.

Let $S'$ be a random subset of $S$ obtained by keeping each vertex of $S$ with probability $\theta/4$. Then $\expn{e_G(S')}\leq \frac{\theta^2}{16}m$ and $\expn{e_G(S',T)}\geq \frac{\theta^2}{8}m$. Using $e_G(T)\leq \frac{\theta^2}{20}m$, it follows that there exists a cut in $G$ with surplus at least $\frac{1}{2}(\frac{\theta^2}{8}m-\frac{\theta^2}{16}m-\frac{\theta^2}{20}m)=\frac{\theta^2}{160}m$, so (ii) holds.

\medskip

\textbf{Case 3.} $e_G(S)\geq (1-\theta)m$.

Let $\tilde{G}=G[S]$ and let $\tilde{n}$ and $\tilde{d}$ be the number of vertices and average degree of $\tilde{G}$, respectively. Clearly, $\tilde{d}\geq (1-\theta)d$, so $\Delta(\tilde{G})\leq C_0d=(1-\theta)Cd\leq C\tilde{d}$. Also, $\tilde{n}\tilde{d}\geq (1-\theta)nd$, so
$$\frac{\tilde{n}^\alpha \tilde{d}^{\beta}}{n^{\alpha} d^{\beta}}\geq \left(\frac{\tilde{n}}{n}\right)^{\alpha} \left((1-\theta)\frac{n}{\tilde{n}}\right)^{\beta}=(1-\eps)\left(\frac{n}{\tilde{n}}\right)^{\beta-\alpha}\geq 1-\eps.$$
Hence, $\tilde{G}$ satisfies (iii).
\end{proof}

It is not hard to deduce Lemma~\ref{lemma:regularize basic}.

\lateproof{Lemma~\ref{lemma:regularize basic}}
By Lemma~\ref{lemma:regularize tight}, applied with $\eps=1/2$, there exist positive constants $c_1$ and $C'$ such that either $G$ has surplus at least $c_1n^{\alpha}d^{\beta}$ or $G$ has an induced subgraph $G'$ with $n'$ vertices and average degree $d'$ where $(n')^{\alpha}(d')^{\beta}\geq \frac{1}{2}n^{\alpha}d^{\beta}$ and $\Delta(G')\leq C'd'$. In the former case we are done. In the latter case, note that $G'$ has an induced subgraph $\tilde{G}$ with at least $e(G')/2$ edges and minimum degree at least $d'/4$. Then, writing $\tilde{n}$ for the number of vertices and $\tilde{d}$ for the average degree of $\tilde{G}$, we have that $\tilde{d} \geq d'/2$ and $\tilde{n}\tilde{d} \geq n'd'/2$. Therefore
$\tilde{n}^{\alpha}\tilde{d}^{\beta}\geq 2^{-\beta}(n')^{\alpha}(d')^{\beta}$, so we can choose $c_2=2^{-\beta-1}$ and $C=4C'$.
\endproof

\subsection{Fewer triangles than in the random graph}

In this subsection, we prove Theorem~\ref{thm:triangle simple}. Using Lemma~\ref{lemma:regularize tight} (as we show in the end of the subsection), it suffices to consider graphs whose maximum degree is not much larger than the average degree. Under this assumption, we prove the following slightly stronger result, which we will use to prove Theorem~\ref{thm:clique optimal}.

\begin{theorem} \label{thm:triangle square threshold}
	Let $G$ be a graph with $n$ vertices, average degree $d$, maximum degree at most $Cd$ and less than $(1-\eps)\frac{d}{6n}\sum_{u\in V(G)} d(u)^2$ triangles, for some $\eps>0$. Then $G$ has surplus $\Omega(\frac{\eps^2}{C}d^2)$.
\end{theorem}

\begin{proof}
Let $\mu=\frac{\eps}{4C}$ and $k=\lceil\mu n/d\rceil$.

First, we give every vertex a random label from $\Set{0,1,\dots,k+1}$, where $0$ is chosen with probability $1/3$, each of $1,\dots,k$ is chosen with probability $d/3n$, and the remaining probability falls on $k+1$. This is feasible since $k\cdot \frac{d}{3n}\leq (\mu n/d+1)\frac{d}{3n}\leq \frac{\mu}{3}+\frac{d}{3n}\leq 2/3$. For every $j\in [k]$, let $B_j$ be the set of vertices with label~$j$.

Now, pick uniformly at random with repetition vertices $v_1,\dots,v_k$.
For $j\in[k]$, let $A_j$ be the set of neighbours of $v_j$ with label $0$ which are not neighbours of any other $v_i$.

Let $X$ be the number of edges which go between $A_j$ and $B_j$ for some $j\in [k]$, let $Y$ be the number of edges inside some $A_j$ and let $Z$ be the number of edges inside some $B_j$.
Our goal will be to show that $\expn{X-Y-Z}= \Omega(\frac{\eps^2}{C}d^2)$.
This is achieved in three separate claims.
\begin{claim}
	$\expn{Y} \le (1-\eps) \frac{kd}{18n^2}\sum_{u\in V(G)} d(u)^2 $.
\end{claim}

\claimproof
	Consider an edge $uv$. Note that $uv$ is only going to be an internal edge of some $A_j$ if some common neighbour is among the $v_j$'s. The probability of this is at most $k\frac{d(u,v)}{n}$. Moreover, both $u$ and $v$ need to be labelled $0$. Now, summing over all edges, we see that $\expn{Y}\le \frac{k}{9n} \sum_{uv\in E(G)}d(u,v)$. The sum counts every triangle exactly three times. Using the assumption, we get the claim.
\endclaimproof

\begin{claim} \label{claim2}
	$\expn{Z} \le \frac{kd^3}{18n}$.
\end{claim}

\claimproof
Consider an edge $uv$. The probability that it is contained in $B_j$ for some fixed $j$ is $(d/3n)^2$. Multiplying this with $k$ for the number of possibilities for $j$ and $nd/2$ for the total number of edges yields the claim.
\endclaimproof

\begin{claim}
	$\expn{X} \ge (1-\eps/4)\frac{kd}{9n^2}\sum_{u\in V(G)} d(u)^2$.
\end{claim}

\claimproof
Consider an edge $uv$. With probability $kd/9n$, $u$ is labelled $0$ and $v$ gets a label from $[k]$. Condition on such a label $j\in [k]$. 
Now, the probability that $v_j$ is a neighbour of $u$, but no other $v_i$ is a neighbour of $u$, is $\frac{d(u)}{n}(1-d(u)/n)^{k-1}\ge \frac{d(u)}{n}(1-(k-1)d(u)/n)\geq \frac{d(u)}{n}(1-(k-1)Cd/n)\geq \frac{d(u)}{n}(1-C\mu)=\frac{d(u)}{n}(1-\eps/4)$.
The same applies with the roles of $u$ and $v$ swapped. Hence, the probability of $uv$ contributing to $X$ is at least $\frac{kd}{9n^2}(d(u)+d(v))(1-\eps/4)$.
Summing over all edges, we obtain the desired inequality.
\endclaimproof

By convexity, $d^2\leq \frac{1}{n}\sum_{u\in V(G)} d(u)^2$, so combining the three claims, we get
$$\expn{X-Y-Z}\geq \frac{kd}{9n^2} \left(\sum_{u\in V(G)} d(u)^2\right)\left((1-\eps/4)-(1-\eps)/2-1/2\right)=\frac{\eps kd}{36n^2}\sum_{u\in V(G)} d(u)^2\geq \frac{\eps \mu}{36}d^2.$$
We infer that there exists an outcome for which $X-Y-Z \ge \Omega(\eps \mu d^2)$.
By Lemma~\ref{lemma:surplus additive}, we conclude that
$$\sp(G)\geq \sum_{j=1}^k \sp(G[A_j\cup B_j])\geq \frac{1}{2}\sum_{j=1}^k (e(A_j,B_j)-e(A_j)-e(B_j))=\frac{1}{2}(X-Y-Z)\geq \Omega(\eps \mu d^2),$$
which completes the proof by using $\mu=\frac{\eps}{4C}$.
\end{proof}

It is easy to deduce Theorem~\ref{thm:triangle simple}.

\lateproof{Theorem~\ref{thm:triangle simple}}
By Lemma~\ref{lemma:regularize tight} (with $\eps/3$ in place of $\eps$, $\alpha=0$ and $\beta=2$), either $G$ has surplus $\Omega_{\eps}(d^2)$, or it has an induced subgraph $\tilde{G}$ with $\tilde{n}$ vertices and average degree $\tilde{d}$, where $\tilde{d}^2\geq (1-\eps/3)d^2$ and $\Delta(\tilde{G})\leq C\tilde{d}$ for some $C=C(\eps)$. In the latter case, we have $\tilde{d}^3\geq (1-\eps/3)^{3/2}d^3\geq (1-\eps/2)d^3$, so the number of triangles in $\tilde{G}$ is less than $(1-\eps)d^3/6\leq \frac{1-\eps}{1-\eps/2}\tilde{d}^3/6\leq (1-\eps/2)\tilde{d}^3/6$. Clearly $\sum_{u\in V(\tilde{G})} d_{\tilde{G}}(u)^2\geq \tilde{n}\tilde{d}^2$, so we can apply Theorem~\ref{thm:triangle square threshold} with $\eps/2$ in place of $\eps$ and deduce that $\tilde{G}$ has surplus $\Omega_{\eps}(\tilde{d}^2)\geq \Omega_{\eps}(d^2)$. It follows that $G$ has the desired surplus.
\endproof

\subsection{\texorpdfstring{$C_5$}{C5}-free graphs} \label{subsec:five cycle}

In this subsection, we prove the following result, which is a special case of Theorem~\ref{thm:odd cycles}. The proof already contains many of the ideas needed to prove the more general result, but is less technical.

\begin{theorem} \label{thm:five cycle}
    Let $G$ be a regular $C_5$-free graph with $m$ edges. Then $G$ has surplus $\Omega(m^{6/7})$.
\end{theorem}

The following lemma will be convenient to use. (Here we only need it for $r=5$, but we will use the general version later.)

\begin{lemma} \label{lemma:five cycle sparse}
    Let $r\geq 3$ and let $G$ be a $d$-degenerate $C_r$-free graph with $m$ edges. Then $\sp(G)=\Omega_r(\frac{m}{\sqrt{d}})$.
\end{lemma}

Since the proof is very standard, we just briefly sketch it. As $G$ is $C_r$-free, the neighbourhood of any vertex induces a $P_{r-2}$-free graph. In particular, the number of edges in $G[N(v)]$ is $O_r(d(v))$, so the number of triangles in $G$ is at most $O_r(m)$. Using Corollary~1.2 from \cite{CKLMST:18} or Lemma~3.3 from \cite{AKS:05} (mentioned in Section~\ref{subsec:intro triangles}), we conclude that $\sp(G)=\Omega_r(\frac{m}{\sqrt{d}})$.

\lateproof{Theorem~\ref{thm:five cycle}}
Let $d$ be the degree of $G$ and let $n$ be the number of vertices in $G$. First note that we may assume that $d$ is arbitrarily large (when $d$ is constant, the theorem follows from Lemma~\ref{lemma:five cycle sparse}). Then clearly, the number of paths of length $2$ in $G$ is $\Omega(nd^2)$.

Here comes one of the key ideas of the proof. By dyadic pigeonholing, there exists a positive integer $s$ such that for at least $\Omega(\frac{nd^2}{s^{1/10}})$ paths $uvw$ in $G$, we have $s\leq d(u,w)<2s$. (The choice of $1/10$ for the exponent is unimportant; it could be replaced by something smaller.) We now prove the following two lower bounds:
\begin{equation}
    \sp(G)=\Omega(nd^{1/2}s^{2/5}) \label{eqn:five cycle SDP}
\end{equation}
and
\begin{equation}
    \sp(G)=\Omega\left(\frac{d^3}{s^{6/5}}\right). \label{eqn:five cycle nhoods}
\end{equation}
Combining these bounds, we can conclude that
$$\sp(G)\geq \Omega\left((nd^{1/2}s^{2/5})^{6/7}\left(\frac{d^3}{s^{6/5}}\right)^{1/7}\right)=\Omega(n^{6/7}d^{6/7}s^{6/35})\geq \Omega(m^{6/7}).$$
Before proving (\ref{eqn:five cycle SDP}) and (\ref{eqn:five cycle nhoods}), let us briefly discuss the intuition behind them. We know that there are many paths $uvw$ with $d(u,w)\approx s$. Assume for the moment that most paths of length~2 are like that. Then the second neighbourhood of a typical vertex has size roughly $d^2/s$. Since $G$ is $C_5$-free, these neighbourhoods induce very sparse subgraphs. We can use these very sparse subgraphs to construct a large cut, just like in Theorem~\ref{thm:triangle square threshold}. Note that we can take roughly $k=\frac{ns}{d^2}$ more-or-less disjoint second neighbourhoods of size roughly $d^2/s$. Each of them is very sparse whereas in a random graph with average degree $d$ these sets would induce roughly $(d^2/s)^2 (d/n)=d^5/(s^2n)$ edges. This results in a surplus roughly $k\cdot d^5/(s^2n)=d^3/s$, which is what (\ref{eqn:five cycle nhoods}) gives, up to a small error term.

The inequality (\ref{eqn:five cycle SDP}) is harder to explain, but observe that a natural type of graph with the property that paths $uvw$ have $d(u,w)\geq s$ is an $s$-blowup of another graph. However, when $G$ is the $s$-blowup of $H$, then $G$ has larger surplus than what is directly provided by Lemma~\ref{lemma:five cycle sparse}. Indeed, $H$ has $m/s^2$ edges and degeneracy at most $d/s$, so by Lemma~\ref{lemma:five cycle sparse}, $\sp(H)=\Omega(\frac{m}{d^{1/2}s^{3/2}})$, and hence $\sp(G)=s^2\sp(H)=\Omega(\frac{m s^{1/2}}{d^{1/2}})=\Omega(nd^{1/2}s^{1/2})$, coinciding with (\ref{eqn:five cycle SDP}), up to a small error term. We now proceed with the proofs of (\ref{eqn:five cycle SDP}) and (\ref{eqn:five cycle nhoods}).

\paragraph{The lower bound $\Omega(nd^{1/2}s^{2/5})$.}

We use the method inspired by the semidefinite relaxation of the MaxCut problem which we presented in Section~\ref{subsec:SDP}. For a vertex $v\in V(G)$, let $S(v)$ be the set of vertices $u\neq v$ with $s\leq d(u,v)<2s$. Now for every $v\in V(G)$, define $\vx^v\in \bR^{V(G)}$ by
$$\vx^v_u=
\begin{cases}
-\frac{1}{\sqrt{d}} \text{ if } u\in N(v)\setminus S(v), \\
\frac{\sqrt{s}}{d} \text{ if } u\in S(v)\setminus N(v), \\
-\frac{1}{\sqrt{d}}+\frac{\sqrt{s}}{d} \text{ if } u\in N(v)\cap S(v), \\
0 \text{ otherwise.}
\end{cases}
$$
Then for any $u,v\in V(G)$,
\begin{align}
    \langle \vx^u,\vx^v\rangle
    &=|N(u)\cap N(v)|\frac{1}{d}+|S(u)\cap S(v)|\frac{s}{d^2} -(|N(u)\cap S(v)|+|S(u)\cap N(v)|)\frac{s^{1/2}}{d^{3/2}}. \label{eqn:five cycle inner product}
\end{align}
In particular, for any $u\in V(G)$, $\|\vx^u\|^2\leq |N(u)|/d+|S(u)|s/d^2$. Clearly $|N(u)|\leq d$. Moreover, $|S(u)|s$ is a lower bound for the number of paths of the form $uvw$ since for any $w\in S(u)$ there are at least $s$ such paths. Therefore,
\begin{equation}
    |S(u)|s\leq d^2, \label{eqn:upper bound on S(u)}
\end{equation}
so we get $\|\vx^u\|^2\leq 2$.
Hence, using (\ref{eqn:five cycle inner product}), and Corollary~\ref{cor:general SDP},
\begin{align}
    \sp(G)
    &\geq \Omega\left(\sum_{uv\in E(G)} (|N(u)\cap S(v)|+|S(u)\cap N(v)|)\frac{s^{1/2}}{d^{3/2}}\right) \nonumber \\
    &-O\left(\sum_{uv\in E(G)} \left(|N(u)\cap N(v)|\frac{1}{d}+|S(u)\cap S(v)|\frac{s}{d^2}\right)\right). \label{eqn:five cycle arcsin}
\end{align}
Observe that
\begin{equation}
    \sum_{uv\in E(G)} |N(u)\cap N(v)|=3t(G)\leq O(nd), \label{eqn:five cycle nhood intersections}
\end{equation}
where in the inequality we used that for any $v\in V(G)$, the graph $G[N(v)]$ is $P_3$-free.

Write $\hom(C_5,G)$ for the number of graph homomorphisms $C_5\rightarrow G$. In other words, $\hom(C_5,G)$ is the number of 5-tuples $(y_1,y_2,y_3,y_4,y_5)\in V(G)^5$ with $y_5y_1\in E(G)$ and $y_iy_{i+1}\in E(G)$ for every $1\leq i\leq 4$. Note that if $w\in S(u)\cap S(v)$ and $uv\in E(G)$, then there are at least $s^2$ such 5-tuples of the form $(u,v,y_3,w,y_5)$. Thus,
$\sum_{uv\in E(G)} |S(u)\cap S(v)|\leq \frac{\hom(C_5,G)}{s^2}$.
On the other hand, $G$ does not contain a 5-cycle. Hence, in any such 5-tuple, we must have $y_i=y_j$ for some $i\neq j$. Without loss of generality, assume that $y_1=y_4$. Then $y_1$, $y_2$ and $y_3$ form a triangle in $G$. Since $y_5$ is a neighbour of $y_1$ and $y_4=y_1$, the number of homomorphic $C_5$'s in $G$ is at most $O(t(G)\cdot \Delta(G))$. As we noted in (\ref{eqn:five cycle nhood intersections}), $t(G)=O(nd)$, so $\hom(C_5,G)=O(nd^2)$ and hence
\begin{equation}
    \sum_{uv\in E(G)} |S(u)\cap S(v)|\leq O\left(\frac{nd^2}{s^2}\right). \label{eqn:five cycle intersection}
\end{equation}

Finally, recall that the number of paths $uvw$ in $G$ with $s\leq d(u,w)<2s$ is $\Omega(\frac{nd^2}{s^{1/10}})$. Any such path has $w\in S(u)\cap N(v)$. Thus,
\begin{equation}
    \sum_{uv\in E(G)} (|S(u)\cap N(v)|+|S(v)\cap N(u)|)\geq \Omega\left(\frac{nd^2}{s^{1/10}}\right). \label{eqn:five cycle mixed intersection}
\end{equation}
Plugging in inequalities (\ref{eqn:five cycle nhood intersections}), (\ref{eqn:five cycle intersection}) and (\ref{eqn:five cycle mixed intersection}) to (\ref{eqn:five cycle arcsin}), we get
$$\sp(G)\geq \Omega(nd^{1/2}s^{2/5})-O(n)-O(n/s)\geq \Omega(nd^{1/2}s^{2/5}).$$

\paragraph{The lower bound $\Omega(d^3/s^{6/5})$.}
This part of the proof is somewhat similar to the proof of Theorem~\ref{thm:triangle square threshold}. Let $\nu nd^2$ be the number of paths $uvw$ in $G$ with $s\leq d(u,w)<2s$. Recall that $\nu=\Omega(s^{-1/10})$. Let $k=\lceil \frac{sn}{2 d^2}\rceil$. Note that $\binom{n}{2}2s$ is an upper bound for the number of paths $uvw$ with $d(u,w)< 2s$, so $\nu nd^2\leq n^2s$ and $\nu d^2/(sn)\leq 1$. Give every vertex a random label from $\{0,1,\dots,k+1\}$, where $0$ is chosen with probability $1/3$, each of $1,\dots,k$ is chosen with probability $p=\frac{\nu d^2}{10sn}$, and the remaining probability falls on $k+1$. This is feasible since $kp\leq (\frac{sn}{2 d^2}+1)\frac{\nu d^2}{10sn}\leq \frac{\nu}{20}+\frac{1}{10}\leq 2/3$. For every $j\in [k]$, let $B_j$ be the set of vertices with label~$j$.

Now, pick uniformly at random with repetition vertices $v_1,\dots,v_k$.
For $j\in[k]$, let $A_j$ be the set of vertices in $S(v_j)$ with label $0$ which do not belong to any other $S(v_i)$.

Let $X$ be the number of edges which go between $A_j$ and $B_j$ for some $j\in [k]$, let $Y$ be the number of edges inside some $A_j$ and let $Z$ be the number of edges inside some $B_j$.
\begin{claim}
    $\expn{Y}\leq O(\frac{kd^2}{s^2})$.
\end{claim}
\claimproof
If an edge $uv\in E(G)$ is in some $A_j$, then $v_j\in S(u)\cap S(v)$. Hence, the probability that $uv$ contributes to $Y$ is at most $k\frac{|S(u)\cap S(v)|}{n}$. Thus,
$$\expn{Y}\leq \sum_{uv\in E(G)} k\frac{|S(u)\cap S(v)|}{n}\leq O\left(\frac{kd^2}{s^2}\right),$$
where for the second inequality we used (\ref{eqn:five cycle intersection}). 
\endclaimproof
\begin{claim}
    $\expn{Z}\leq \frac{kp^2nd}{2}$.
\end{claim}
\claimproof
For any $j$, the expected number of edges in $G[B_j]$ is $p^2e(G)=p^2\frac{nd}{2}$. Summing over the $k$ possibilities for $j$, the claim follows.
\endclaimproof
\begin{claim}
    $\expn{X}\geq \frac{kp\nu d^3}{12s}$.
\end{claim}
\claimproof
Let $uv\in E(G)$. For any $j$, the probability that $v_j\in S(u)$, but $v_i\not \in S(u)$ for every $j\neq i$ is
$$\frac{|S(u)|}{n}\left(1-\frac{|S(u)|}{n}\right)^{k-1}\geq \frac{|S(u)|}{n}\left(1-\frac{(k-1)|S(u)|}{n}\right)\geq \frac{|S(u)|}{n}\left(1-\frac{(k-1)d^2}{sn}\right)\geq \frac{|S(u)|}{2n},$$
where the second inequality uses (\ref{eqn:upper bound on S(u)}).
Thus, the probability that $uv$ contributes to $X$ is at least $k\frac{p}{6n}(|S(u)|+|S(v)|)$. Summing over all edges,
$$\expn{X}\geq \frac{kpd}{6n}\sum_{u\in V(G)} |S(u)|.$$
Since the number of paths $uvw$ with $s\leq d(u,w)<2s$ is $\nu nd^2$, there are at least $\nu nd^2/(2s)$ pairs $u,w$ with $w\in S(u)$, so
$\sum_{u\in V(G)}|S(u)|\geq \nu nd^2/(2s)$, from which the claim follows.
\endclaimproof
Combining the three claims, we get
$$\expn{X-Y-Z}\geq kd\left(\frac{p\nu d^2}{12s}-O(d/s^2)-p^2n/2\right)=kd\left(\frac{\nu^2 d^4}{120s^2n}-O(d/s^2)-\frac{\nu^2 d^4}{200 s^2n}\right).$$
If $\frac{\nu^2d^4}{s^2n}=O(d/s^2)$, then $d^3=O(n/\nu^2)\leq O(ns^{1/5})$, so the bound $\sp(G)\geq \Omega(n\sqrt{d}) \geq \Omega(d^3/s^{6/5})$ follows from Lemma~\ref{lemma:five cycle sparse}. Otherwise,
$$\expn{X-Y-Z}\geq \Omega\left(\frac{k\nu^2d^5}{s^2n}\right)\geq \Omega\left(\frac{\nu^2 d^3}{s}\right)\geq \Omega\left(\frac{d^3}{s^{6/5}}\right).$$
Therefore there exists an outcome in which $X-Y-Z\geq \Omega(d^3/s^{6/5})$, and then
$$\sp(G)\geq \sum_{j=1}^k \sp(G[A_j\cup B_j])\geq \frac{1}{2}\sum_{j=1}^k (e(A_j,B_j)-e(A_j)-e(B_j)) = \frac{1}{2}(X-Y-Z)\geq \Omega(d^3/s^{6/5}),$$
completing the proof.
\endproof

\section{Odd cycles} \label{sec:odd cycles}

In this section, we prove Theorem~\ref{thm:odd cycles}. As discussed in Section~\ref{subsec:odd cycles}, the tightness of the result follows from the known construction of dense pseudorandom $C_r$-free graphs due to Alon and Kahale~\cite{AK:98}.
It remains to show that, for fixed odd $r\ge 3$, any $C_r$-free graph with $m$ edges has surplus $\Omega_r(m^{(r+1)/(r+2)})$.
Since the case $r=3$ is covered by Alon's theorem, we may assume $r\ge 5$.
The proof is similar to the one for $C_5$ which we presented in Section~\ref{sec:illustration}, but more technical, and some new ideas are needed.

In the ``sparse case'', when the graph is $d$-degenerate, we have the bound $\Omega(m/\sqrt{d})$ on the surplus. The rest of this section is devoted to the proof of the following lemma which we use in the ``dense'' case.
We set 
\begin{align}
	\alpha=\frac{r}{r+1} \mbox{ and } \beta=\frac{2r+1}{2r+2}.\label{exponents}
\end{align}
Note that $0<\alpha<\beta$ and $\alpha+\beta\le 2$.

\begin{lemma} \label{lemma:general cycle dense}
	Let $G$ be a $C_r$-free graph with $n$ vertices and average degree~$d$. Then $\sp(G)=\Omega_r(n^{\alpha}d^{\beta})$.
\end{lemma}

Before proving it, let us see how it implies Theorem~\ref{thm:odd cycles}. We will need the following lemma.

\begin{lemma} \label{lemma:good partition}
    Let $G$ be a graph and let $d$ be a non-negative real number. Then there exists a bipartition $(S,T)$ of $V(G)$ such that $G[S]$ is $d$-degenerate and $G[T]$ has minimum degree at least~$d$.
\end{lemma}

\begin{proof}
Start with $S_0=\emptyset$ and $T_0=V(G)$. If there exists a vertex $v\in T_0$ with less than $d$ neighbours in $T_0$, then set $S_1=S_0\cup \{v\}$ and $T_1=T_0\setminus \{v\}$. More generally, for $i\geq 0$, if there exists $v\in T_i$ with less than $d$ neighbours in $T_i$, then let $S_{i+1}= S_i\cup \{v\}$ and let $T_{i+1}=T_i\setminus \{v\}$. Eventually the process stops and we are left with a partition $(S_k,T_k)$. Take $S=S_k$, $T=T_k$. By definition, $G[S]$ is $d$-degenerate and $G[T]$ has minimum degree at least~$d$.
\end{proof}

\lateproof{Theorem~\ref{thm:odd cycles}}
Let $G$ be a $C_r$-free graph with $m$ edges. Let $d=m^{2/(r+2)}$ and take a partition $(S,T)$ of $V(G)$ as in Lemma~\ref{lemma:good partition}.
If $e(S,T)\geq 2m/3$, then clearly $\sp(G)\geq m/6$. If $e(S)\geq m/6$, then by Lemma~\ref{lemma:five cycle sparse}, $$\sp(G)\ge \sp(G[S])\geq \Omega_r\left(\frac{m/6}{\sqrt{d}}\right)=\Omega_r(m^{1-1/(r+2)})$$ as desired. Finally, suppose that $e(T)\geq m/6$ and write $n$ for the number of vertices in $G[T]$. Clearly, $nd\geq 2e(T)\geq m/3$. Therefore, by Lemma~\ref{lemma:general cycle dense}, $$\sp(G)\ge \sp(G[T])\geq \Omega_r(n^{\alpha}d^{\beta}) = \Omega_r\left(m^\alpha d^{\beta-\alpha}\right)= \Omega_r(m^{\alpha+2(\beta-\alpha)/(r+2)}) = \Omega_r(m^{1-1/(r+2)}),$$ since $\alpha+2(\beta-\alpha)/(r+2)= 1-1/(r+2)$.
\endproof

It remains to prove Lemma~\ref{lemma:general cycle dense}.
By Lemma~\ref{lemma:regularize basic} (applied with $\alpha,\beta$ as defined in~\eqref{exponents}), it suffices to prove Lemma~\ref{lemma:general cycle dense} for $C_r$-free graphs with $n$ vertices, average degree $d$, minimum degree $\Omega_r(d)$ and maximum degree $O_r(d)$. 
Assume throughout that $G$ is such a graph. We may of course also assume that $d$ is sufficiently large.

Let $\eps>0$ be chosen sufficiently small, depending only on~$r$. (Hence, implicit constants depending on $\eps$ will also be indicated by a subscript~$r$.)
Set $$\ell=\frac{r-1}{2}.$$
Recall from Section~\ref{subsec:five cycle} that it was crucial to find many $2$-paths $uvw$ for which the codegree of $u$ and $w$, that is, the number of $2$-paths from $u$ to $w$, was roughly the same. Here, we need to generalize this concept to paths of length~$\ell$. For some steps of our argument, it will actually be more convenient to work with walks rather than paths.
For two vertices $u,v\in V(G)$ and an integer $j$, we write $h_j(u,v)$ for the number of walks in $G$ of length $j$ from $u$ to~$v$.

Using the minimum degree condition, we easily observe that there are $\Omega_r(nd^\ell)$ paths of length $\ell$ in~$G$.
Now, the first crucial step in our argument is to select a large subset of $\ell$-paths which exhibit a certain regularity condition.

\begin{prop}\label{prop:pigeonholing}
	There exist $c=c(r)>0$, positive integers $(s_{i,j})_{0\le i<j\le \ell}$ and at least $$c\left(\prod_{i<j}s_{i,j}\right)^{-\eps} nd^\ell$$ paths $u_0u_1\dots u_\ell$ in $G$ with the property that 
	\begin{align}
		s_{i,j} \le h_{j-i}(u_i,u_j)< 2s_{i,j}\label{good}
	\end{align}
	for all $0\le i<j\le \ell$.
\end{prop}

\begin{proof}
	We apply dyadic pigeonholing. For each given path $u_0u_1\dots u_\ell$, there are unique non-negative integers $(b_{i,j})_{0\le i<j\le \ell}$ such that $2^{b_{i,j}} \le h_{j-i}(u_i,u_j)< 2^{b_{i,j}+1}$. 

	Suppose, for a contradiction, that for every choice of $(b_{i,j})_{0\le i<j\le \ell}$, there are at most $$c\left(\prod_{i<j}2^{b_{i,j}}\right)^{-\eps} nd^\ell$$ corresponding paths. Summing over all choices for the $b_{i,j}$, we deduce that the total number of $\ell$-paths is at most $cnd^\ell \left(\sum_{b=0}^\infty 2^{-\eps b} \right)^{\binom{\ell+1}{2}}=cnd^\ell O_r(1)$. Choosing $c$ small enough, we reach a contradiction.
\end{proof}

From now on, we fix $c$ and $(s_{i,j})_{0\le i<j\le \ell}$ which satisfy the conclusion of Proposition~\ref{prop:pigeonholing}, and call the paths $u_0u_1\dots u_\ell$ in $G$ which satisfy~\eqref{good} \defn{good}. Hence, there are at least $\Omega_r\left(\nu nd^\ell\right)$ good $\ell$-paths in~$G$, where we set $$\nu=\left(\prod_{i<j}s_{i,j}\right)^{-\eps}.$$
Note that we clearly have $s_{i,j}\le \Delta(G)^{j-i-1} = d^{O(\ell)}$ for all $0\le i<j\le \ell$. Thus, for small enough~$\eps$, we have, say, 
\begin{align}
\nu\ge d^{-1/10}. \label{nu bound}
\end{align}

We extend the notion of a good path to smaller lengths in the obvious way. That is, for $q\le \ell$, a path $u_0u_1\dots u_q$ is called \defn{good} if~\eqref{good} holds for all $0\le i<j\le q$.
Since $G$ has maximum degree $O_r(d)$, every $q$-path can be extended into at most $O_r(d^{\ell-q})$ $\ell$-paths. Moreover, every subpath of a good $\ell$-path (with the same start vertex) is also good. Hence, for each $q\le \ell$, there are at least $\Omega_r\left(\nu nd^q\right)$ good $q$-paths in~$G$.

We now state the key lemmas which provide lower bounds on the surplus.
Using the SDP (semidefinite programming) method, we will prove the following.

\begin{lemma} \label{lemma:SDPsurplus}
	For each $2\leq q\leq \ell$, we have
	$$\sp(G) = \Omega_r\left(\nu nd^{1/2} (s_{0,q}/s_{0,q-1})^{1/2}\right).$$
\end{lemma}

\begin{cor} \label{cor:SDPsurplus}
	We have
	$$\sp(G) = \Omega_r\left(\nu nd^{1/2} s_{0,\ell}^{1/2(\ell-1)}\right).$$
\end{cor}

\begin{proof}
By the AM--GM inequality,
$$\frac{1}{\ell-1} \sum_{q=2}^\ell (s_{0,q}/s_{0,q-1})^{1/2} \ge \left(\prod_{q=2}^\ell (s_{0,q}/s_{0,q-1})^{1/2}\right)^{1/(\ell-1)} = s_{0,\ell}^{1/2(\ell-1)}$$
since $s_{0,1}=1$. We infer that there exists $2\leq q\leq \ell$ with $(s_{0,q}/s_{0,q-1})^{1/2}\ge s_{0,\ell}^{1/2(\ell-1)}$.
Lemma~\ref{lemma:SDPsurplus} then implies the claim.
\end{proof}

Moreover, using random neighbourhood sampling, we will prove the following bound on the surplus. We actually only need this for $q=\ell$.

\begin{lemma} \label{lemma:nhoodsurplus}
	For each $2\leq q\leq \ell$, we have
	$$\sp(G) = \Omega_r\left(\nu^2 d^{q+1}/s_{0,q}\right).$$
\end{lemma}

Before proving the two key lemmas, let us complete the proof of Lemma~\ref{lemma:general cycle dense}. We need one final simple observation.
\begin{prop}\label{prop:sis}
	For any $0\le i\le i'<j'\le j \le \ell$, we have $s_{i',j'}\le 2s_{i,j}$. In particular, $\nu = (2s_{0,\ell})^{-O(\ell^2\eps)}$.
\end{prop}

\begin{proof}
	Consider any good $\ell$-path $u_0u_1\dots u_\ell$. Then we know $h_{j-i}(u_i,u_j)\le 2s_{i,j}$. On the other hand, we have $h_{j-i}(u_i,u_j) \ge h_{j'-i'}(u_{i'},u_{j'}) \ge s_{i',j'}$ since we can replace $u_{i'}\dots u_{j'}$ with any walk of length $j'-i'$ from $u_{i'}$ to $u_{j'}$. Hence, $s_{i',j'}\le 2s_{i,j}$, proving the first claim.
	The second claim easily follows.
\end{proof}

\lateproof{Lemma~\ref{lemma:general cycle dense}}
Write $s=s_{0,\ell}$. We have
$$\max\Set{d^{\ell+1}/s, nd^{1/2} s^{1/2(\ell-1)} } \ge \left(\frac{d^{\ell+1}}{s}\cdot \left(nd^{1/2} s^{1/2(\ell-1)}\right)^{2\ell+1}\right)^{\frac{1}{2\ell+2}}= n^{\alpha}d^{\beta}s^{\frac{3}{(r-3)(r+1)}}. $$
Thus, by Corollary~\ref{cor:SDPsurplus} and Lemma~\ref{lemma:nhoodsurplus}, we have $\sp(G)=\Omega_r(\nu^2 n^{\alpha}d^{\beta}s^{\frac{3}{(r-3)(r+1)}})$.
Finally, by Proposition~\ref{prop:sis}, we can ensure that $\nu^2 s^{\frac{3}{(r-3)(r+1)}}=\Omega(1)$ when $\eps$ is sufficiently small, which completes the proof.
\endproof

\begin{remark}
Observe that the same proof, using $\max\Set{a,b}\ge (ab^{2\ell+2})^{1/(2\ell+3)}$ instead of $\max\Set{a,b}\ge (ab^{2\ell+1})^{1/(2\ell+2)}$, gives us $\sp(G)=\Omega_r\left((nd)^{(2\ell+2)/(2\ell+3)}\right)$, which proves Theorem~\ref{thm:odd cycles} directly when $G$ is (almost) regular. We used the slightly biased weighting here in order to be able to apply our regularization lemma.
\end{remark}

It remains to prove the key Lemmas~\ref{lemma:SDPsurplus} and~\ref{lemma:nhoodsurplus}. Since some steps in the proofs are similar, we treat them in a unified setup. 
For the rest of this section, fix $2\le q\le \ell$. To enhance readability, we also set $s=s_{0,q} \text{ and } s'=s_{0,q-1}$.

\paragraph{Partition of $V(G)$.} We will be interested in good paths of length~$q$. In order to maintain better control on how different such paths can intersect, we consider ``partite'' paths where the vertices are chosen from disjoint vertex subsets.
To this end, partition $V(G)$ randomly as $U_0\cup U_1\cup \dots \cup U_q$ by placing, each vertex $u$ independently, in $U_i$ with probability $1/(q+1)$, for every $0\leq i\leq q$. Write $\cA$ for the set of tuples $(u_0,u_1,\dots,u_q)\in U_0\times U_1\times \dots \times U_q$ such that $u_0u_1\dots u_q$ is a good path.
Write $\cB$ for the set of tuples $(u_0,u_1,\dots,u_{q-1})\in U_0\times U_1\times \dots \times U_{q-1}$ such that $u_0u_1\dots u_{q-1}$ is a good path.

Since there are $\Omega_r(\nu nd^q)$ good paths of length $q$ in $G$, each of which contributes to $\cA$ with probability $(q+1)^{-(q+1)}=\Omega_r(1)$, we have $\expn{|\cA|}= \Omega_r(\nu nd^q)$.
Hence, there exists a partition $U_0\cup U_1\cup \dots \cup U_q$, which we will fix from now on, for which 
$|\cA|=\Omega_r(\nu nd^q)$. (Note that, since any tuple in $\cB$ can be extended to at most $\Delta(G)$ tuples in $\cA$, we also know that $|\cB|=\Omega_r(\nu nd^{q-1})$, although we will not explicitly need this.)

\paragraph{The sets $S(u),T(u)$.} 
For every vertex $u\in V(G)$, define sets $T(u)$ and $S(u)$ as follows. If $u\in U_0\cup U_1\cup \dots \cup U_{q-2}$, then let $T(u)=S(u)=\emptyset$. If $u\in U_{q-1}$, then let $S(u)=\emptyset$ and let $T(u)$ be the set of those $u_0 \in U_0$ for which there exist at least $\frac{|\cA|s'}{4n\Delta(G)^q}=\Omega_r(\nu s')$ tuples of the form $(u_0,u_1,u_2,\dots,u_{q-2},u) \in \cB$. Finally, if $u\in U_q$, then let $T(u)=\emptyset$ and let $S(u)$ be the set of those $u_0\in U_0$ for which there exist at least $\frac{|\cA|s}{4n\Delta(G)^q}=\Omega_r(\nu s)$ tuples of the form $(u_0,u_1,\dots,u_{q-1},u) \in \mathcal{A}$.

\bigskip
Very roughly, we can think of $S(u)$ as the $q$th neighbourhood of $u$, but we only use this for $u\in U_q$ and only consider ``$q$th neighbours'' in $U_0$ which can be reached by a substantial amount of good paths.

The following lemma is the crucial point where we use the assumption that $G$ is $C_r$-free. We will use it in the proofs of both key lemmas in order to bound terms that would otherwise drive the surplus to be small.

\begin{lemma}\label{lem:intersection sums}
We have $$\sum_{uv\in E(G)} |S(u)\cap S(v)| = O_r\left(\frac{nd^q}{\nu s}\right)\text{ and } 
\sum_{uv\in E(G)} |T(u)\cap T(v)| = O_r\left(\frac{nd^{q-1}}{\nu s'}\right).$$
\end{lemma}

The intuition is that if $u_0\in S(u)\cap S(v)$, then we must have $u,v\in U_q$, and two of the good paths in $\cA$, one from $u_0$ to $u$ and one from $u_0$ to $v$, will give us together with the edge $uv$ an odd cycle. In the extreme case that $q=\ell$ and the two paths are internally vertex-disjoint, this would be a cycle of length $2\ell+1=r$, a contradiction to our assumption that $G$ is $C_r$-free. 
In order to prove the more general case when $q$ can be smaller and the paths might not be internally disjoint, we need a counting lemma which we will establish in the next subsection.

\subsection{Counting good paths and short cycles}

As discussed above, when $u_0\in S(u)\cap S(v)$, then we find an odd cycle of the form $u_i\dots u_{q-1}uvv_{q-1}\dots v_i$, where $u_i=v_i$ and $u_j,v_j\in U_j$ for all $i\le j\le q-1$. The following lemma shows that we can utilize the $C_r$-freeness of $G$ to control such short odd cycles. In its statement, one can think of $z=u_i=v_i$ and $R\In U_q$.

\begin{lemma} \label{lemma:smallavgdeg}
	Let $z\in V(G)$ and let $R \subset V(G)\setminus \{z\}$. Let $j \leq \ell$ be a positive integer. Define the auxiliary graph $H$ with vertex set $R$ such that $uv\in E(H)$ if there is a $C_{2j+1}$ in $G$ of the form $zu_1u_2\dots u_{j-1} u v v_{j-1}v_{j-2}\dots v_1z$ with $u_1,\dots,u_{j-1},v_1,\dots,v_{j-1}\not \in R$. Then $H$ has average degree at most $2^{2j+1}(r-2j)=O_r(1)$.
\end{lemma}

\begin{proof}
	Define a random partition $V(G)\setminus \{z\}=A\cup B$ by placing each vertex on either side with probability $1/2$, independently of the other vertices. Let $H'$ be the subgraph of $H$ on the same vertex set $R$ in which $ab$ is an edge if $a\in A$, $b\in B$ and there is a $C_{2j+1}$ in $G$ of the form $zu_1u_2\dots u_{j-1}abv_{j-1}v_{j-2}\dots v_1z$ such that $u_1,\dots,u_{j-1}\in A\setminus R$ and $v_1,\dots,v_{j-1}\in B\setminus R$. Clearly, each edge in $H$ is an edge in $H'$ with probability at least $2^{-2j}$. Hence, $\mathbb{E}\lbrack e(H')\rbrack\geq 2^{-2j}e(H)$, and there exists a partition $A,B$ for which $e(H')\ge 2^{-2j}e(H)$.
	
	 Crucially, $H'$ cannot contain a path of length $r-2j$. Indeed, suppose there is such a path $a_1b_1a_2b_2\dots a_ib_i$ where $r-2j=2i-1$. Without loss of generality, $a_1\in A$ and $b_i\in B$. Since $a_1b_1\in E(H')$, there exists a path $zu_1u_2\dots u_{j-1}a_1$ in $G$ with $u_1,u_2,\dots,u_{j-1}\in A\setminus R$. Similarly, there exists a path $zv_1v_2\dots v_{j-1}b_i$ in $G$ with $v_1,v_2,\dots,v_{j-1}\in B\setminus R$. Then $$zu_1 u_2\dots u_{j-1}a_1b_1a_2b_2\dots a_ib_iv_{j-1}\dots v_1z$$ is a $C_r$ in $G$, which is a contradiction. Thus, $H'$ does not contain a path of length $r-2j$. It follows that $H'$ has average degree at most $2(r-2j)$. Hence, $H$ has average degree at most $2^{2j+1}(r-2j)=O_r(1)$.
\end{proof}

\begin{lemma} \label{lemma:shortpathcount}
	Let $z\in V(G)$ and let $R \subset V(G)\setminus \{z\}$. Let $j \leq \ell$ be a positive integer. Suppose there exists a real number $s''\ge 1$ such that for any $u\in R$, we have $s''\leq h_j(z,u)< 2s''$.
Then there are at most $O_r(d^j)$ paths $zu_1\dots u_{j-1}uv$ in $G$ with $u_1,\dots,u_{j-1}\not \in R$, $u,v\in R$ which can be extended to a $C_{2j+1}$ of the form $zu_1\dots u_{j-1}uv v_{j-1}\dots v_1z$ with $v_1,\dots,v_{j-1}\not \in R$.
\end{lemma}

\begin{proof}
	Since every $u\in R$ has $h_j(z,u)\geq s''$ and the total number of walks of length $j$ starting from $z$ is at most $\Delta(G)^j=O_r(d^j)$, we have $|R|=O_r(d^j/s'')$. Define the auxiliary graph $H$ like in Lemma~\ref{lemma:smallavgdeg}. Then we have $e(H)= O_r(d^j/s'')$. Since $h_j(z,u)< 2s''$ for every $u\in R$, there are at most $2e(H)\cdot 2s''=O_r(d^j)$ paths $zu_1\dots u_{j-1}uv$ in $G$ with $u_1,\dots,u_{j-1}\not \in R$ and $u,v\in R$ which can be extended to a $C_{2j+1}$ of the form $zu_1\dots u_{j-1}uv v_{j-1}\dots v_1z$ with $v_1,\dots,v_{j-1}\not \in R$.
\end{proof}

\begin{lemma} \label{lemma:countcircuits}
	The number of tuples $(u_0,\dots,u_q,v)$ with $(u_0,\dots,u_q)\in \cA$, $v\in N(u_q)\cap U_q$ and $u_0\in S(v)$ is at most $O_r(nd^q)$. Similarly, the number of tuples $(u_0,\dots,u_{q-1},v)$ with $(u_0,\dots,u_{q-1})\in \cB$, $v\in N(u_{q-1})\cap U_{q-1}$ and $u_0\in T(v)$ is at most $O_r(nd^{q-1})$. 
\end{lemma}

\begin{proof}
	We prove the first claim. The proof of the second is almost verbatim the same.
	Take a tuple $(u_0,\dots,u_q,v)$ with $(u_0,\dots,u_q)\in \cA$ and $v\in N(u_q)\cap U_q$ for which $u_0\in S(v)$. The last condition implies that there exists $(v_0,v_1,\dots,v_{q-1},v)\in \cA$ with $v_0=u_0$.
	Let $i$ be the largest index for which $u_i=v_i$. Then $u_i\dots u_{q-1}u_qvv_{q-1}\dots v_i$ is a $C_{2j+1}$ in $G$, where $1\le j=q-i\le q\le \ell$. We now count the number of choices for $(u_i,u_{i+1},\dots,u_{q-1},u_q,v)$.
	
	There are at most $n$ possibilities for~$u_i$. Fix such a choice $u_i$ (which also determines~$i$). Set $R=\set{u'\in U_q}{s_{i,q}\le h_{q-i}(u_i,u')<2s_{i,q}}$.	
	By Lemma~\ref{lemma:shortpathcount} (with $z=u_i=v_i$ and $s''=s_{i,q}$), there are at most $O_r(d^j)$ choices for $(u_{i+1},\dots,u_{q-1},u_q,v)$. Given any such choice, there are clearly at most $\Delta(G)^i=O_r(d^{q-j})$ possibilities to choose $u_{i-1},\dots,u_0$. In total, this gives $O_r(nd^q)$ possibilities. 
\end{proof}

\lateproof{Lemma~\ref{lem:intersection sums}}
We prove the first bound, the proof of the second is verbatim the same.
Note that $\sum_{uv\in E(G)} |S(u)\cap S(v)|$ is half the number of triples $(u_0,u,v)$ with $u,v\in U_q$, $uv\in E(G)$ and $u_0\in S(u)\cap S(v)$. Each such triple, spelling out $u_0\in S(u)$, gives us $\Omega_r(\nu s)$ tuples $(u_0,\dots,u_{q-1},u,v)$ with $(u_0,\dots,u_{q-1},u)\in \cA$, $v\in N(u)\cap U_q$ and $u_0\in S(v)$.

Thus, it suffices to show that the total number of tuples $(u_0,\dots,u_{q-1},u,v)$ with $(u_0,\dots,u_{q-1},u)\in \cA$, $v\in N(u)\cap U_q$ and $u_0\in S(v)$ is at most $O_r(nd^q)$. This follows from Lemma~\ref{lemma:countcircuits}.
\endproof

One last tool is the following, which will help us in both key lemmas to give good bounds on terms that influence the surplus in our favour.

\begin{prop}\label{prop:good paths}
	At least half of the tuples $(u_0,\dots,u_{q-1},u_q)\in\cA$ satisfy $u_0\in T(u_{q-1})\cap S(u_q)$.
\end{prop}
\begin{proof}
	We count the number of tuples $(u_0,\dots,u_{q-1},u_q)\in\cA$ with $u_0\notin S(u_{q})$ as follows. First, we have at most $n$ choices for $u_0$. Then, the number of possibilities for $u_q$ such that there exists at least one tuple $(u_0,u_1,\dots,u_q)\in\cA$ is at most $\Delta(G)^q/s$, since there are obviously at most $\Delta(G)^q$ walks of length $q$ starting from $u_0$, and if the final vertex $u_q$ is such that there exists a good $q$-path from $u_0$ to $u_q$, then we have $h_q(u_0,u_q)\ge s$. 
	Now, given $u_0$ and $u_q$, if $u_0\notin S(u_{q})$, then there are at most $\frac{|\cA|s}{4n\Delta(G)^q}$ tuples of the form $(u_0,\dots,u_{q-1},u_q)\in\cA$ by definition of $S(u_{q})$. Hence, the total number of such tuples is at most $|\cA|/4$.
	
	The same proof shows that the number of $(u_0,\dots,u_{q-1})\in\cB$ with $u_0\notin T(u_{q-1})$ is at most $$n\cdot \frac{\Delta(G)^{q-1}}{s'} \cdot \frac{|\cA|s'}{4n\Delta(G)^q} = \frac{|\cA|}{4\Delta(G)}.$$
	Since any tuple in $\cB$ can be extended to at most $\Delta(G)$ tuples in $\cA$, we see that the number of tuples $(u_0,\dots,u_{q-1},u_q)\in\cA$ with $u_0\notin T(u_{q-1})$ is also at most $|\cA|/4$, which completes the proof.
\end{proof}

\subsection{Lower bound using the SDP method}

Our goal in this subsection is to prove Lemma~\ref{lemma:SDPsurplus}.
We need one more tool which controls the gain that we will achieve with the SDP method.

\begin{prop}\label{prop:intersection sums gain}
	We have $$\sum_{uv\in E(G)} |S(u)\cap T(v)| = \Omega_r\left(\frac{\nu n d^q}{s'}\right),$$ where the sum here is over all ordered pairs $(u,v)$ with $uv\in E(G)$.
\end{prop}

\begin{proof}
By Proposition~\ref{prop:good paths}, there are $\Omega_r(\nu nd^q)$ tuples $(u_0,\dots,u_{q-1},u_q)\in\cA$ which satisfy $u_0\in T(u_{q-1})\cap S(u_q)$. Each of these gives a triple $(u_0,u_{q-1},u_q)$ which is counted in $\sum_{uv\in E(G)} |S(u)\cap T(v)|$. Moreover, any such triple $(u_0,u_{q-1},u_q)$ arises in this way at most $2s'$ times since, if some $(u_0,u_1,\dots,u_{q-1},u_q)\in\cA$ exists, then the number of choices for $u_1,\dots,u_{q-2}$ is at most $h_{q-1}(u_0,u_{q-1})< 2s'$.
\end{proof}

\lateproof{Lemma~\ref{lemma:SDPsurplus}}
For every $u\in V(G)$, define $\vx^u\in \bR^{V(G)}$ by
	$$\vx^u_w=
	\begin{cases}
		- (s'/d^{q-1})^{1/2} \text{ if } w\in T(u), \\
		(s/d^q)^{1/2} \text{ if } w\in S(u), \\
		0 \text{ otherwise.}
	\end{cases}
	$$
	This is well-defined since for any $u\in V(G)$, we have $T(u)\cap S(u)=\emptyset$.
	
	Our goal is to apply Corollary~\ref{cor:general SDP}. Clearly, we have
	$||\vx^u||^2 \le |T(u)|s'/d^{q-1} + |S(u)|s/d^q$.
	For each $u_0\in T(u)$, we know that $h_{q-1}(u_0,u)\ge s'$. Hence, the number of walks of length $q-1$ starting from $u$ is at least $|T(u)|s'$. On the other hand, this quantity is trivially bounded from above by $\Delta(G)^{q-1}$. Hence, we have $|T(u)|s'/d^{q-1}=O_r(1)$. Similarly, the number of walks of length $q$ starting from $u$ is at least $|S(u)|s$ and at most $\Delta(G)^{q}$, implying that $|S(u)|s/d^{q}=O_r(1)$, too. Hence, $||\vx^u||=O_r(1)$ for all $u\in V(G)$.
	
	Now, consider $uv\in E(G)$. We have  $\langle \vx^u,\vx^v\rangle = -a_{uv}+b_{uv}$, where
	\begin{align*}
        a_{uv} &= (|T(u)\cap S(v)|+|S(u)\cap T(v)|)\frac{(ss')^{1/2}}{d^{q-1/2}} ,
        \\
        b_{uv} &= |T(u)\cap T(v)|\frac{s'}{d^{q-1}} +|S(u)\cap S(v)|\frac{s}{d^q}.
	\end{align*}
Invoking Proposition~\ref{prop:intersection sums gain}, we have a ``gain'' of
$$\sum_{uv\in E(G)} a_{uv} = \Omega_r\left( \frac{\nu nd^q}{s'} \frac{(ss')^{1/2}}{d^{q-1/2}} \right) =  \Omega_r\left(\nu nd^{1/2} (s/s')^{1/2}\right).$$
On the other hand, by Lemma~\ref{lem:intersection sums}, we have a ``loss'' of
	$$\sum_{uv\in E(G)} b_{uv} = O_r\left(\frac{nd^{q-1}}{\nu s'}\frac{s'}{d^{q-1}}+\frac{nd^q}{\nu s} \frac{s}{d^{q}} \right) =O_r(n/\nu).$$

Since $s'\le 2s$ by Proposition~\ref{prop:sis} and $\nu\ge d^{-1/10}$ by~\eqref{nu bound}, the second sum is negligible compared to the first, and so by Corollary~\ref{cor:general SDP}, we conclude $\sp(G)= \Omega_r\left(\nu nd^{1/2} (s/s')^{1/2}\right)$.
This completes the proof.
\endproof

Note that in the above proof, $\vx^u$ is the zero vector unless $u\in U_{q-1}\cup U_q$. This essentially means that we find a cut with large surplus in $G[U_{q-1}\cup U_q]$.

\subsection{Random neighbourhood sampling}

Our goal in this subsection is to prove Lemma~\ref{lemma:nhoodsurplus}. While in the proof of Lemma~\ref{lemma:SDPsurplus}, the dominant term for the surplus was given by the sum over all edges $uv$ of the intersection size $|S(u)\cap T(v)|$, here, the surplus will be mainly given by $\sum_{u\in V(G)} |S(u)| $. Recall that $S(u)=\emptyset$ unless $u\in U_q$, so the important vertices here are the ones in~$U_q$. Roughly speaking, we find subsets $A_j$ inside $U_q$ which contain few edges, and use this to find a large cut.

\begin{prop}\label{prop:nhoodgain}
We have $\sum_{u\in V(G)} |S(u)| =  \Omega_r(\nu nd^q/s)$.
\end{prop}

\begin{proof}
The sum $\sum_{u\in V(G)} |S(u)|$ counts the number of pairs $(u_0,u_q)\in U_0\times U_q$ with $u_0\in S(u_q)$. By Proposition~\ref{prop:good paths} and $\cA \geq \Omega_r(\nu nd^q)$, and since each pair $(u_0,u_q)\in U_0\times U_q$ appears in at most $2s$ tuples in $\cA$, we have $\sum_{u\in V(G)} |S(u)| = \Omega_r(\nu nd^q/s)$.
\end{proof}

\lateproof{Lemma~\ref{lemma:nhoodsurplus}}
Let $\mu$ be a sufficiently small constant depending only on~$r$.
Note that $|\cA|\le n^2 (2s)$ since for any choice of $(u_0,u_q)\in U_0\times U_q$ that appears in some tuple in $\cA$, there are at most $2s$ tuples containing it. Moreover, recall $|\cA|=\Omega_r(\nu nd^q)$. Hence, $\frac{\nu d^q}{sn}=O_r(1)$.
Let $$k=\lcl \frac{sn}{2\Delta(G)^q}\rcl.$$
Give every vertex a random label from $\{0,1,\dots,k+1\}$, where $0$ is chosen with probability $1/3$, each of $1,\dots,k$ is chosen with probability $p=\mu\frac{\nu d^q}{sn}$, and the remaining probability falls on $k+1$. This is feasible since $kp\leq \mu/2 +p \leq 2/3$, provided $\mu$ is sufficiently small. For every $j\in [k]$, let $B_j$ be the set of vertices with label~$j$.

Now, pick uniformly at random with repetition vertices $v_1,\dots,v_k\in V(G)$.
For $j\in[k]$, let $A_j$ be the set of vertices $u$ with label $0$ such that $v_j\in S(u)$, but $v_i\notin S(u)$ for all $i\neq j$.

Let $X$ be the number of edges which go between $A_j$ and $B_j$ for some $j\in [k]$, let $Y$ be the number of edges inside some $A_j$ and let $Z$ be the number of edges inside some $B_j$.
\begin{claim}
	$\expn{Y}=O_r(\frac{kd^q}{\nu s})$.
\end{claim}
\claimproof
If an edge $uv\in E(G)$ is in some $A_j$, then $v_j\in S(u)\cap S(v)$. Hence, the probability that $uv$ contributes to $Y$ is at most $k\frac{|S(u)\cap S(v)|}{n}$. Thus,
$$\expn{Y}\leq \sum_{uv\in E(G)} k\frac{|S(u)\cap S(v)|}{n}= O_r\left(\frac{kd^q}{\nu s}\right)$$
by Lemma~\ref{lem:intersection sums}.
\endclaimproof

\begin{claim}
	$\expn{Z}=O(kp^2nd)$.
\end{claim}
\claimproof
For any $j$, the expected number of edges in $G[B_j]$ is $p^2e(G)=p^2\frac{nd}{2}$. Summing over the $k$ possibilities for $j$, the claim follows.
\endclaimproof

\begin{claim}
	$\expn{X}=\Omega_r(kp\nu d^{q+1}/s)$.
\end{claim}
\claimproof
Let $uv\in E(G)$. For any $j\in[k]$, the probability that $v_j\in S(u)$, but $v_i\not \in S(u)$ for every $i\neq j$, is
$$\frac{|S(u)|}{n}\left(1-\frac{|S(u)|}{n}\right)^{k-1}\geq \frac{|S(u)|}{n}\left(1-\frac{(k-1)|S(u)|}{n}\right)\geq \frac{|S(u)|}{n}\left(1-\frac{(k-1)\Delta(G)^q}{sn}\right)\geq \frac{|S(u)|}{2n},$$
where we used the bound $|S(u)|s\le \Delta(G)^q$ (explained in the proof of Lemma~\ref{lemma:SDPsurplus}) and the definition of~$k$.
Thus, the probability that $uv$ contributes to $X$ is at least $k\frac{p}{6n}(|S(u)|+|S(v)|)$. Hence,
$$\expn{X}\geq \frac{kp}{6n}\sum_{u\in V(G)} d(u)|S(u)|\geq \frac{kp\delta(G)}{6n}\sum_{u\in V(G)} |S(u)| =  \Omega_r(kp\nu d^{q+1}/s),$$
where we have used Proposition~\ref{prop:nhoodgain} in the last step.
\endclaimproof

Combining the three claims, and plugging in $p$ with a sufficiently small $\mu$, we get
$$\expn{X-Y-Z}\geq kd \left(\Omega_r(p\nu d^{q}/s)-O_r\left(\frac{d^{q-1}}{\nu s}\right) -O(p^2n)  \right) \ge kd\left(\Omega_r\left(\frac{\nu^2 d^{2q}}{s^2n}\right)-O_r\left(\frac{d^{q-1}}{\nu s}\right) \right).$$

If $\frac{\nu^2 d^{2q}}{s^2n}=O_r\left(\frac{d^{q-1}}{\nu s}\right)$, then $\frac{\nu^2d^{q+1}}{s}=O_r(n/\nu)=O_r(nd^{1/10})$ by~\eqref{nu bound}, so the surplus we are aiming for already follows from Lemma~\ref{lemma:five cycle sparse}.
Otherwise,
$$\expn{X-Y-Z}= \Omega_r\left(\frac{k\nu^2 d^{2q+1}}{s^2n}\right) = \Omega_r\left(\frac{\nu^2 d^{q+1}}{s}\right).$$
Then there exists an outcome in which $X-Y-Z= \Omega_r\left(\nu^2 d^{q+1}/s\right)$, and therefore
$$\sp(G)\geq \sum_{j=1}^k \sp(G[A_j\cup B_j])\geq \frac{1}{2}(X-Y-Z)=\Omega_r\left(\nu^2 d^{q+1}/s\right)$$
by Lemma~\ref{lemma:surplus additive}, completing the proof.
\endproof

\section{Few triangles} \label{sec:few triangles}

In this section, we prove Theorem~\ref{thm:triangledependence}. Similarly to the proof of Theorem~\ref{thm:strongly regular surplus}, we let $a=\frac{\sqrt{d}}{n-d}$, $0<\gamma\leq 1$ and for every $v\in V(G)$, we define $\vx^v\in \bR^{V(G)}$ by
$$\vx^v_u=
\begin{cases}
1+\gamma a \text{ if } u=v, \\
-\gamma \frac{1}{\sqrt{d}} \text{ if } u\in N(v), \\
\gamma a \text{ otherwise.}
\end{cases}
$$

Just like in the strongly regular case, the sum $\sum_{uv\in E(G)} \frac{\langle \vx^u,\vx^v\rangle}{\|\vx^u\|\|\vx^v\|}$ is negative for a suitable choice of $\gamma$. However, this does not directly lead to a large surplus because the inner product is not negative for every $uv\in E(G)$, and for positive $x$, we have $\arcsin(x)>x$. Nevertheless, for small positive $x$, $\arcsin(x)$ is very close to $x$, so the main discrepancy is caused by edges $uv$ for which $\langle \vx^u,\vx^v\rangle$ is very large. In view of (\ref{eqn:innerproduct}), this happens when $d(u,v)$ is much larger than $d^2/n$. To deal with these edges, we will introduce a slightly modified version of $\vx^v$. For $v\in V(G)$, define $\vy^v\in \mathbb{R}^{V(G)}$ randomly by
$$\vy^v_u=
\begin{cases}
1+\gamma a \text{ if } u=v, \\
-\gamma \frac{1}{\sqrt{d}} \text{ if } u\in N(v) \text{ and } d(u,v)\leq 20d^2/n, \\
\pm \gamma \frac{1}{\sqrt{d}} \text{ if } u\in N(v) \text{ and } d(u,v)> 20d^2/n, \\
\gamma a \text{ otherwise,}
\end{cases}
$$
where, for any $uv\in E(G)$ with $d(u,v)>20d^2/n$, the sign of $\vy^v_u$ is + with probability $1/2$, independently of all other coordinates and vectors.

Our first lemma states that if there are many edges with large codegree, then, when applying Lemma~\ref{lemma:general SDP}, the vectors $\vy^v$ provide significantly better surplus than the vectors $\vx^v$.

\begin{lemma} \label{lemma:compare x and y}
    Let $0<\gamma\leq 1/10$ and $d\leq n/2$. If $G$ is an $n$-vertex $d$-regular graph with at most $d^3/3$ triangles and the vectors $\vx^v$, $\vy^v$ are defined as above, then
    $$\expn{\sum_{uv\in E(G)} \left(\arcsin\left(\frac{\langle \vy^u,\vy^v\rangle}{\|\vy^u\| \|\vy^v\|}\right)-\arcsin\left(\frac{\langle \vx^u,\vx^v\rangle}{\|\vx^u\| \|\vx^v\|}\right)\right)} \leq \frac{\gamma n\sqrt{d}}{4}-\gamma^2\left(\sum_{\substack{uv\in E(G): \\ d(u,v)>20d^2/n}} \frac{d(u,v)}{10d}\right).$$
\end{lemma}

The proof of this lemma is a tedious calculation, so it is given in the appendix.
Choose signs in a way that
$$\sum_{uv\in E(G)} \left(\arcsin\left(\frac{\langle \vy^u,\vy^v\rangle}{\|\vy^u\| \|\vy^v\|}\right)-\arcsin\left(\frac{\langle \vx^u,\vx^v\rangle}{\|\vx^u\| \|\vx^v\|}\right)\right) \leq \frac{\gamma n\sqrt{d}}{4}-\gamma^2\left(\sum_{\substack{uv\in E(G): \\ d(u,v)>20d^2/n}} \frac{d(u,v)}{10d}\right).$$
By Lemma~\ref{lemma:general SDP}, we get
\begin{align}
    \sp(G)
    &\geq -\frac{1}{\pi}\sum_{uv\in E(G)} \arcsin\left(\frac{\langle \vy^u,\vy^v \rangle}{\|\vy^u\|\|\vy^v\|}\right) \nonumber \\
    &\geq -\frac{1}{\pi}\left(\sum_{uv\in E(G)} \arcsin\left(\frac{\langle \vx^u,\vx^v \rangle}{\|\vx^u\|\|\vx^v\|}\right)+\frac{\gamma n\sqrt{d}}{4}-\gamma^2\left(\sum_{\substack{uv\in E(G): \\ d(u,v)>20d^2/n}} \frac{d(u,v)}{10d}\right)\right). \label{eqn:sp with pi}
\end{align}

We can use this inequality to prove the following lemma, which is a lower bound on the surplus of a graph in terms of the number of triangles and the codegrees. Recall that for vertices $u,v$, we write $\delta(u,v)=d(u,v)-d^2/n$. Let $\delta_+(u,v)$ be equal to $\delta(u,v)$ when $\delta(u,v)>0$, and otherwise let $\delta_+(u,v)=0$.

\begin{lemma} \label{lemma:sp with deltaplus}
    Let $0<\gamma\leq 1/10$ and $d\leq n/2$. If $G$ is an $n$-vertex $d$-regular graph with $d^3/6+s$ triangles where $s\leq d^3/6$, then
    $$\sp(G)\geq \Theta(\gamma n\sqrt{d})-\Theta\left(\frac{\gamma^2 s}{d}\right)-\Theta\left(\gamma^2 \sum_{\substack{uv\in E(G): \\ d(u,v)\leq 20d^2/n}} \frac{\delta_+(u,v)^3}{d^3}\right).$$
\end{lemma}

The proof is again given in the appendix.
Our goal is now to remove the ``error term" in the above lemma and prove the following stronger result.

\begin{lemma} \label{lemma:sp without deltaplus}
    Let $0<\gamma\leq 1/10$ and $d\leq n/2$. If $G$ is an $n$-vertex $d$-regular graph with $d^3/6+s$ triangles where $s\leq d^3/6$, then
    $\sp(G)\geq \Theta(\gamma n\sqrt{d})-\Theta\left(\frac{\gamma^2 s}{d}\right)$.
\end{lemma}

This is a significant sharpening (in the regular case) of Corollary 1.2 from \cite{CKLMST:18}, which states that if $G$ is a $d$-degenerate graph with $m$ edges and $t$ triangles, then for any $\gamma\leq 1$, we have $\sp(G)\geq \frac{\gamma m}{2\pi \sqrt{d}}-\frac{\gamma^2 t}{2d}$. Note that $s$ is much smaller than $t$, and can be negative. Given Lemma~\ref{lemma:sp without deltaplus}, it is easy to deduce Theorem~\ref{thm:triangledependence}.

\lateproof{Theorem~\ref{thm:triangledependence}}
We claim that for any $0<\gamma\leq 1/10$,
\begin{equation}
    \sp(G)\geq \Theta(\gamma n\sqrt{d})-\Theta\left(\frac{\gamma^2 s}{d}\right). \label{eqn:sp with gamma}
\end{equation}
When $s\leq d^3/6$, this follows from Lemma~\ref{lemma:sp without deltaplus}. Otherwise, $s=\Theta(t(G))$ and we can use the aforementioned Corollary 1.2 from \cite{CKLMST:18} to obtain the desired inequality.

When $s<-nd^{3/2}$, take $\gamma=1/10$ and use (\ref{eqn:sp with gamma}) to get $\sp(G)= \Omega(|s|/d)$.

When $-nd^{3/2}\leq s\leq nd^{3/2}$, take $\gamma$ to be a sufficiently small constant and use (\ref{eqn:sp with gamma}) to get $\sp(G)= \Omega(n\sqrt{d})$.

Finally, when $s>nd^{3/2}$, take $\gamma=c\frac{nd^{3/2}}{s}$ for a sufficiently small constant $c$ and use (\ref{eqn:sp with gamma}) to get $\sp(G)=\Omega(n^2d^2/s)$.
\endproof

The rest of this section is devoted to the proof of Lemma~\ref{lemma:sp without deltaplus}.

\begin{lemma} \label{lemma:surplusinaux}
    Let $H$ be a graph on $N$ vertices. Let $D>0$ be a real number. For $z\in V(H)$, let $\Delta_+(z)=d_H(z)-D$ if $d_H(z)>D$, and otherwise let $\Delta_+(z)=0$. Then there exists a positive absolute constant $c$ with the following property. If \begin{equation}
        e(H)\leq \frac{ND}{2}+\frac{c}{D^2}\sum_{\substack{z\in V(H): \\ d_H(z)\leq 20D}} \Delta_+(z)^3, \label{eqn:bound on edges of H}
    \end{equation}
    then there exists some $T\subset V(H)$ such that $$e_H(T)\leq \frac{|T|^2}{2}\frac{D}{N}-\frac{c}{D^2}\sum_{\substack{z\in V(H): \\ d_H(z)\leq 20D}} \Delta_+(z)^3.$$
\end{lemma}

\begin{proof}
Let $c=\frac{1}{12\cdot 40^2}$ and let $H$ be a graph on $N$ vertices satisfying (\ref{eqn:bound on edges of H}).
Write $$q=\sum_{\substack{z\in V(H): \\ d_H(z)\leq 20D}} \Delta_+(z)^3.$$ Then
\begin{equation}
    e(H)\leq \frac{ND}{2}+\frac{c}{D^2}q. \label{eqn:upper bound on e(H)}
\end{equation}
Observe that there exists a positive integer $i$ such that $$\sum_{z:\,2^{-i}20D< \Delta_+(z)\leq 2^{-(i-1)}20D} \Delta_+(z)^3\geq q/2^{i}.$$
Letting $S$ be the set of $z\in V(H)$ with $2^{-i}20D< \Delta_+(z)\leq 2^{-(i-1)}20D$, we get
\begin{equation}
    |S|\geq \frac{q/2^{i}}{(2^{-(i-1)}20D)^3}= \frac{2^{2i}q}{40^3D^3}. \label{eqn:lower bound on S}
\end{equation}
Let $p=2^{-i}$ and let $S'$ be the random subset of $S$ obtained by keeping each vertex of $S$ with probability $p$, independently of the other vertices. Let $T=V(H)\setminus S'$.

\begin{claim}
\begin{equation}
    \mathbb{E}[e_H(T)]\leq \frac{\mathbb{E}[|T|^2]}{2}\frac{D}{N}-\frac{c}{D^2}q. \label{eqn:expectedgain}
\end{equation}
\end{claim}
This suffices since then we can find $T\subset V(H)$ such that $e_H(T)\leq\frac{|T|^2}{2}\frac{D}{N}-\frac{c}{D^2}q$. \claimproof
For a set $X\subset V(H)$, write $e(X)$ for the number of edges in $H[X]$. For disjoint sets $X$ and $Y$, write $e(X,Y)$ for the number of edges between $X$ and $Y$ in $H$.

We start by giving an upper bound for $\expn{e(T)}$. Let $R=V(H)\setminus S$. Since $T=V(H)\setminus S'$, we have $e(T)=e(H)-e(T,S')-e(S')=e(H)-e(R,S')-e(S\setminus S',S')-e(S')$. Taking expectations, we get
\begin{align}
\expn{e(T)}
&=e(H)-\expn{e(R,S')}-\expn{e(S\setminus S',S')}-\expn{e(S')} \nonumber \\
&=e(H)-pe(R,S)-2(1-p)pe(S)-p^2e(S) \nonumber \\
&=e(H)-p(e(R,S)+2e(S))+p^2e(S) \nonumber \\
&=e(H)-p\left(\sum_{z\in S} d_H(z)\right)+p^2e(S) \nonumber \\
&\leq e(H)-p|S|(D+2^{-i}20D)+p^2e(S) \nonumber \\
&\leq e(H)-p|S|D-p2^{-i}|S|20D+p^2\frac{|S|21D}{2} \nonumber \\
&\leq e(H)-p|S|D-\frac{1}{3}2^{-2i}|S|20D, \nonumber \\
&\leq \frac{ND}{2}+\frac{c}{D^2}q-p|S|D-\frac{1}{3}2^{-2i}|S|20D, \label{eqn:bound on e(T)}
\end{align}
where the first inequality follows from the fact every $z\in S$ has $\Delta_+(z)>2^{-i}20D$, the second inequality holds since $\Delta_+(z)\leq 2^{-(i-1)}20D\leq 20D$ for every $z\in S$, and the last inequality is true because of (\ref{eqn:upper bound on e(H)}). By (\ref{eqn:lower bound on S}), we have $\frac{1}{3}2^{-2i}|S|20D\geq \frac{q}{6\cdot 40^2 D^2}=\frac{2c}{D^2}q$. Hence, by (\ref{eqn:bound on e(T)}) we have
$$\expn{e(T)}\leq \frac{ND}{2}-p|S|D-\frac{c}{D^2}q.$$

On the other hand, $\expn{|T|^2}\geq \expn{|T|}^2=(N-p|S|)^2\geq N^2-2p|S|N$, so
$$\frac{\expn{|T|^2}}{2}\frac{D}{N}\geq \frac{ND}{2}-p|S|D,$$
which completes the proof of the claim.
\end{proof}

Finally, we need the following lemma.

\begin{lemma} \label{lemma:surplusfromsparse}
    Let $G$ be a $d$-regular graph on $n$ vertices and let $S\subset V(G)$ be a subset. Then $\sp(G)\geq \frac{1}{2}\left(\frac{|S|^2}{2}\frac{d}{n}-e_G(S)\right)$.
\end{lemma}

\begin{proof}
Let $T$ be a random subset of $V(G)$ such that each $v\in V(G)$ belongs to $T$ with probability $|S|/n$, independently of the other vertices. Write $e_G(S,T)$ for the number of pairs $(u,v)\in S\times T$ with $uv\in E(G)$. Clearly, $\mathbb{E}[e_G(S,T)]=|S|d\cdot \frac{|S|}{n}$. Moreover, $\mathbb{E}[e_G(T)]=\frac{nd}{2}\cdot \frac{|S|^2}{n^2}$. Hence,
$$\mathbb{E}[e_G(S,T)-e_G(S)-e_G(T)]=\frac{|S|^2}{2}\frac{d}{n}-e_G(S).$$
So we may choose $T$ in a way that $e_G(S,T)-e_G(S)-e_G(T)\geq \frac{|S|^2}{2}\frac{d}{n}-e_G(S)$. If $S$ and $T$ are disjoint, then we get the desired surplus in $G[S\cup T]$. Otherwise, one by one for each vertex $v$ which appears both in $S$ and $T$, we can remove $v$ from one of the two remaining sets (to be precise, from the one in which it has more neighbours) and not decrease the surplus corresponding to the cut.
\end{proof}

\lateproof{Lemma~\ref{lemma:sp without deltaplus}}
Let $c$ be the constant provided by Lemma~\ref{lemma:surplusinaux}. If $$\sum_{\substack{uv\in E(G): \\ d_G(u,v)\leq 20d^2/n}} \delta_+(u,v)^3< \frac{3}{2c}sd^2,$$
then the result follows from Lemma~\ref{lemma:sp with deltaplus}. Assume that
\begin{equation}
    \sum_{\substack{uv\in E(G): \\ d_G(u,v)\leq 20d^2/n}} \delta_+(u,v)^3\geq  \frac{3}{2c}sd^2. \label{eqn:deltacubes}
\end{equation}

Define an auxiliary graph $H$ as follows. Let $V(H)=\{(u,v)\in V(G)^2: uv\in E(G)\}$ and take an edge between $(u,v)$ and $(u,w)$ if $vw\in E(G)$. Let all other vertex pairs be non-edges. Observe that $e(H)=3t(G)=d^3/2+3s$. Let $N=|V(H)|=nd$ and $D=\frac{d^2}{n}$. Note that for any $z=(u,v)\in V(H)$, we have $d_H(z)=d_G(u,v)$. Moreover, if $\Delta_+(z)$ is defined as in Lemma~\ref{lemma:surplusinaux}, then for $z=(u,v)$, we have $\Delta_+(z)=\delta_+(u,v)$. Hence, (\ref{eqn:deltacubes}) translates to
$$\sum_{\substack{z\in V(H): \\ d_H(z)\leq 20D}} \Delta_+(z)^3\geq  \frac{3}{c}sd^2,$$
since an edge $uv\in E(G)$ corresponds to both $(u,v)\in V(H)$ and $(v,u)\in V(H)$. Thus,
$$\frac{c}{D^2}\sum_{\substack{z\in V(H): \\ d_H(z)\leq 20D}} \Delta_+(z)^3\geq \frac{3sd^2}{D^2}\geq 3s,$$
and hence
$$e(H)=d^3/2+3s\leq \frac{ND}{2}+\frac{c}{D^2}\sum_{\substack{z\in V(H): \\ d_H(z)\leq 20D}} \Delta_+(z)^3.$$
Therefore, by Lemma~\ref{lemma:surplusinaux}, there exists $T\subset V(H)$ such that
$$e_H(T)\leq \frac{|T|^2}{2}\frac{D}{N}-\frac{c}{D^2}\sum_{\substack{z\in V(H): \\ d_H(z)\leq 20D}} \Delta_+(z)^3.$$
For every $u\in V(G)$, let $T_u=\{(u,v)\in T\}$. Note that $\sum_{u\in V(G)} |T_u|^2\geq \frac{(\sum_{u\in V(G)} |T_u|)^2}{n}=\frac{|T|^2}{n}$. Moreover, there are no edges in $H$ between $T_u$ and $T_v$ when $u\neq v$, so $e_H(T)=\sum_{u\in V(G)} e_H(T_u)$. Thus,
$$\sum_{u\in V(G)} e_H(T_u)\leq \sum_{u\in V(G)} \frac{|T_u|^2}{2}\frac{nD}{N}-\frac{c}{D^2}\sum_{\substack{z\in V(H): \\ d_H(z)\leq 20D}} \Delta_+(z)^3.$$
Therefore there exists $w\in V(G)$ such that
$$e_H(T_w)\leq \frac{|T_w|^2}{2}\frac{nD}{N}-\frac{c}{nD^2}\sum_{\substack{z\in V(H): \\ d_H(z)\leq 20D}} \Delta_+(z)^3=\frac{|T_w|^2}{2}\frac{d}{n}-\frac{2cn}{d^4}\sum_{\substack{uv\in E(G): \\ d_G(u,v)\leq 20d^2/n}} \delta_+(u,v)^3.$$
This means that there exists a set $S\subset N_G(w)$ such that
$$e_G(S)\leq \frac{|S|^2}{2}\frac{d}{n}-\frac{2cn}{d^4}\sum_{\substack{uv\in E(G): \\ d_G(u,v)\leq 20d^2/n}} \delta_+(u,v)^3.$$
Hence, by Lemma~\ref{lemma:surplusfromsparse}, $$\sp(G)\geq \frac{cn}{d^4}\left(\sum_{\substack{uv\in E(G): \\ d_G(u,v)\leq 20d^2/n}} \delta_+(u,v)^3\right) \geq \Omega\left(\gamma^2 \sum_{\substack{uv\in E(G): \\ d(u,v)\leq 20d^2/n}} \frac{\delta_+(u,v)^3}{d^3}\right).$$
This, combined with Lemma~\ref{lemma:sp with deltaplus}, implies the result.
\endproof

\section{Cliques} \label{sec:cliques}

In this section, we prove Theorems~\ref{thm:clique degeneracy}, \ref{thm:clique free in terms of m} and~\ref{thm:clique optimal}.
We start with the proof of Theorem~\ref{thm:clique optimal}, as the proof of the other two results will use its weaker version Theorem~\ref{thm:clique free simple}.

\lateproof{Theorem~\ref{thm:clique optimal}}
We prove the theorem by induction on $r$. The case $r=3$ is Theorem~\ref{thm:triangle simple}. Assume that $r\geq 4$. First observe that we may assume that $G$ has maximum degree at most $Cd$ for some $C$, depending only on $\eps$ and $r$. Indeed, we may apply Lemma~\ref{lemma:regularize tight} with $\eps/r$ in place of $\eps$, $\alpha=-(r-3)$ and $\beta=r-1$ to find an induced subgraph $\tilde{G}$ with $\tilde{n}$ vertices and average degree $\tilde{d}$ such that $\frac{\tilde{d}^{r-1}}{\tilde{n}^{r-3}}\geq (1-\eps/r)\frac{d^{r-1}}{n^{r-3}}$ and $\Delta(\tilde{G})\leq C\tilde{d}$ (if such subgraph does not exist, then $G$ has surplus $\Omega_{\eps,r}(d^{r-1}/n^{r-3})$ by the lemma). Now
$$\tilde{n}^r(\tilde{d}/\tilde{n})^{\binom{r}{2}}=(\tilde{d}^{r-1}/\tilde{n}^{r-3})^{r/2}\geq ((1-\eps/r)d^{r-1}/n^{r-3})^{r/2}= (1-\eps/r)^{r/2}n^r(d/n)^{\binom{r}{2}}\geq (1-\eps/2)n^r(d/n)^{\binom{r}{2}},$$
so the number of $K_r$'s in $\tilde{G}$ is at most $(1-\eps/2)\frac{\tilde{n}^r}{r!}(\tilde{d}/\tilde{n})^{\binom{r}{2}}$.

Thus, by replacing $G$ by $\tilde{G}$ and $\eps$ by $\eps/2$ if necessary, we may assume that $\Delta(G)\leq Cd$ for some $C=C(\eps,r)$.

If $G$ has at most $(1-\frac{\eps}{4r^2})\frac{d}{6n}\sum_{v\in V(G)} d(v)^2$ triangles, then by Theorem~\ref{thm:triangle square threshold}, $G$ has surplus $\Omega_{\eps,r,C}(d^2)$ and we are done since $d^2\ge d^{r-1}/n^{r-3}$. So assume that
\begin{equation}
    t(G)\geq \left(1-\frac{\eps}{4r^2}\right)\frac{d}{6n}\sum_{v\in V(G)} d(v)^2. \label{eqn:triangle count}
\end{equation}
Let $k=\lceil \frac{\eps}{8Cr^2}\frac{n}{d}\rceil$ and let $v_1,v_2,\dots,v_k$ be random vertices chosen from $V(G)$ with replacement. For every $1\leq i\leq k$, let $A_i=N(v_i)\setminus (\cup_{j\neq i} N(v_j))$.

Fix $1\leq i\leq k$. Let $X$ be the number of edges in $G[A_i]$ and let $Y$ be the number of copies of $K_{r-1}$ in $G[A_i]$.

\begin{claim}
    $\expn{X}\geq (1-\frac{\eps}{4r^2})\frac{3t(G)}{n}$. \label{claim:edges in A}
\end{claim}

\claimproof
Let $xy\in E(G)$. We have $xy\in E(G[A_i])$ precisely if $x,y\in N(v_i)\setminus (\cup_{j\neq i} N(v_j))$, i.e. if $v_i\in N(x)\cap N(y)$, but $v_j\not \in N(x)\cup N(y)$ for every $j\neq i$. Since $|N(x)\cup N(y)|\leq  2Cd$, we get
$$\mathbb{P}(xy\in E(G[A_i]))\geq \frac{d(x,y)}{n}\left(1-\frac{2Cd}{n}\right)^{k-1}\geq \frac{d(x,y)}{n}\left(1-(k-1)\frac{2Cd}{n}\right)\geq \left(1-\frac{\eps}{4r^2}\right)\frac{d(x,y)}{n}.$$
Hence,
$$\expn{X}\geq \sum_{xy\in E(G)} \left(1-\frac{\eps}{4r^2}\right)\frac{d(x,y)}{n}= \left(1-\frac{\eps}{4r^2}\right)\frac{3\cdot t(G)}{n}.$$
\endclaimproof

\begin{claim}
    $\expn{Y}\leq (1-\eps)\frac{n^{r-1}}{(r-1)!}(d/n)^{\binom{r}{2}}$. \label{claim:cliques in A}
\end{claim}

\claimproof
Let $x_1x_2\dots x_{r-1}$ be a copy of $K_{r-1}$ in $G$. Then, writing $d(x_1,\dots,x_{r-1})$ for the number of common neighbours of $x_1,\dots,x_{r-1}$, we have
$\mathbb{P}(x_1,\dots,x_{r-1}\in A_i)\leq \frac{d(x_1,\dots,x_{r-1})}{n}$.
Since the sum of $d(x_1,\dots,x_{r-1})$ over all $(r-1)$-cliques $x_1\dots x_{r-1}$ is $r$ times the number of $K_r$'s, and the number of $K_r$'s in $G$ is at most $(1-\eps)\frac{n^r}{r!}(d/n)^{\binom{r}{2}}$, we get the desired bound.
\endclaimproof

For simplicity, write $A=A_i$. Note that $G[A]$ has average degree $\frac{2X}{|A|}$. By the induction hypothesis, there exists a positive $c'=c'(r,\eps)$ such that if $Y\leq (1-\eps/2)\frac{|A|^{r-1}}{(r-1)!}(2X/|A|^2)^{\binom{r-1}{2}}$, then $G[A]$ has surplus at least $c'(2X)^{r-2}/|A|^{2r-6}$. On the other hand, since $|A|^{r-1}(2X/|A|^2)^{\binom{r-1}{2}}=((2X)^{r-2}/|A|^{2r-6})^{\frac{r-1}{2}}$, if
$$Y\geq (1-\eps/2)\frac{|A|^{r-1}}{(r-1)!}(2X/|A|^2)^{\binom{r-1}{2}},$$
then
$$c'\left(\frac{(r-1)!\cdot Y}{1-\eps/2}\right)^{\frac{2}{r-1}}\geq c'(2X)^{r-2}/|A|^{2r-6}.$$
Hence, in both cases we have
\begin{equation*}
    \sp(G[A])\geq c'(2X)^{r-2}/|A|^{2r-6}-c'\left(\frac{(r-1)!\cdot Y}{1-\eps/2}\right)^{\frac{2}{r-1}}.
\end{equation*}
Taking expectations, we get
\begin{equation}
    \expn{\sp(G[A])}\geq c'\left(2^{r-2}\expn{X^{r-2}/|A|^{2r-6}}-\left(\frac{(r-1)! }{1-\eps/2}\right)^{\frac{2}{r-1}}\expn{Y^{\frac{2}{r-1}}}\right). \label{eqn:surplus in A}
\end{equation}
By H\"older's inequality, we have
$$\expn{\frac{X^{r-2}}{|A|^{2r-6}}}^{\frac{1}{r-2}}\expn{|A|^2}^{\frac{r-3}{r-2}}\geq \expn{\left(\frac{X^{r-2}}{|A|^{2r-6}}\right)^{\frac{1}{r-2}}|A|^{\frac{2(r-3)}{r-2}}}=  \expn{X},$$
so
$$\expn{\frac{X^{r-2}}{|A|^{2r-6}}}\geq \frac{\expn{X}^{r-2}}{\expn{|A|^{2}}^{r-3}}.$$
Observe that $\expn{|A|^2}\le \frac{1}{n}\sum_{v\in V(G)} d(v)^2$. This inequality is the reason why we used Theorem~\ref{thm:triangle square threshold} instead of Theorem~\ref{thm:triangle simple} in our argument.
By Claim~\ref{claim:edges in A} and (\ref{eqn:triangle count}),
$$\expn{X}\geq \left(1-\frac{\eps}{4r^2}\right)\frac{3t(G)}{n}\geq \left(1-\frac{\eps}{4r^2}\right)^2\frac{d}{2n^2}\sum_{v\in V(G)} d(v)^2\geq \left(1-\frac{\eps}{4r^2}\right)^2\frac{d}{2n}\expn{|A|^2}.$$
Thus, using Claim~\ref{claim:edges in A}, (\ref{eqn:triangle count}) and $\sum_v d(v)^2\geq nd^2$,
\begin{align*}
    \frac{\expn{X}^{r-2}}{\expn{|A|^2}^{r-3}}
    &\geq \left(\left(1-\frac{\eps}{4r^2}\right)^2\frac{d}{2n}\right)^{r-3}\expn{X}\geq \left(\left(1-\frac{\eps}{4r^2}\right)^2\frac{d}{2n}\right)^{r-3}\left(1-\frac{\eps}{4r^2}\right)\frac{3t(G)}{n} \\
    &\geq \left(\left(1-\frac{\eps}{4r^2}\right)^2\frac{d}{2n}\right)^{r-3}\left(1-\frac{\eps}{4r^2}\right)^2\frac{d^3}{2n}\geq \left(1-\frac{\eps}{2r}\right)\frac{d^r}{(2n)^{r-2}},
\end{align*}
so
\begin{equation}
    \expn{\frac{X^{r-2}}{|A|^{2r-6}}}\geq \left(1-\frac{\eps}{2r}\right)\frac{d^r}{(2n)^{r-2}}. \label{eqn:ratio lower bound}
\end{equation}
On the other hand, by Claim~\ref{claim:cliques in A}, we have
$$\expn{Y^{\frac{2}{r-1}}}\leq \expn{Y}^{\frac{2}{r-1}}\leq \left(\frac{1-\eps}{(r-1)!}\right)^{\frac{2}{r-1}}\frac{d^r}{n^{r-2}}.$$
Substituting this and equation (\ref{eqn:ratio lower bound}) into equation (\ref{eqn:surplus in A}), we get
$$\expn{\sp(G[A])}\geq c'\left(\left(1-\frac{\eps}{2r}\right)\frac{d^r}{n^{r-2}}-\left(\frac{1-\eps}{1-\eps/2}\right)^{\frac{2}{r-1}}\frac{d^r}{n^{r-2}}\right).$$
Note that $(\frac{1-\eps}{1-\eps/2})^{\frac{2}{r-1}}\leq (1-\eps/2)^{\frac{2}{r-1}}\leq 1-\eps/(r-1)\leq 1-\eps/r$, so we conclude that
$$\expn{\sp(G[A])}\geq c'\frac{\eps}{2r}\frac{d^r}{n^{r-2}}.$$
Since $\sp(G)\geq \sum_{i=1}^k \sp(G[A_i])$ by Lemma~\ref{lemma:surplus additive}, we get
$$\sp(G)\geq k\expn{\sp(G[A])}\geq \frac{\eps}{8Cr^2}c'\frac{\eps}{2r}\frac{d^{r-1}}{n^{r-3}},$$
which completes the proof.
\endproof

	It remains to prove Theorems~\ref{thm:clique degeneracy} and~\ref{thm:clique free in terms of m}.
	We prove them simultaneously using induction on~$r$. 
	Carlson et al.~\cite{CKLMST:18} provided a general tool that allows to convert results on the surplus formulated purely in terms of the number of edges (as in Theorem~\ref{thm:clique free in terms of m}) into results that hold for degenerate graphs (as in Theorem~\ref{thm:clique degeneracy}). In particular, a special case of their Lemma~3.5 is the following: If $\sp(m,K_{r-1})=\Omega_r(m^a)$ for some $a=a(r)\in [\frac{1}{2},1]$, then any $d$-degenerate graph $G$ with $m\ge1$ edges has surplus 
	\begin{align}
	\Omega_r(m/d^{(2-a)/(1+a)}). \label{m to degen}
	\end{align}
	
	\lateproof{Theorems~\ref{thm:clique degeneracy} and~\ref{thm:clique free in terms of m}}
	We know that both theorems hold for $r=3$.
	Assume now that $r\ge 4$ and that both theorems hold for~$r-1$. It follows from \eqref{m to degen} (applied with $a=\frac{1}{2}+\frac{3}{4(r-1)-2}=\frac{r}{2r-3}$) that Theorem~\ref{thm:clique degeneracy} holds for~$r$. We now use this additional information to prove Theorem~\ref{thm:clique free in terms of m} for~$r$.
	
	Let $G$ be any $K_r$-free graph with $m$ edges.
Let $d=m^{\frac{r-1}{2r-1}}$. Assume first that $G$ is $d$-degenerate. Then by Theorem~\ref{thm:clique degeneracy}, $G$ has surplus $\Omega_r(md^{-\frac{r-2}{r-1}})=\Omega_r(m^{1-\frac{r-2}{2r-1}})=\Omega_r(m^{\frac{r+1}{2r-1}})$, which is the desired bound. Otherwise, $G$ has a non-empty induced subgraph $H$ with minimum degree at least $d$. If the number of vertices in $H$ is $n$, then clearly $nd\leq 2e(H)\leq 2e(G)=2m$, so $n\leq \frac{2m}{d}=2m^{\frac{r}{2r-1}}$. Thus, by Theorem~\ref{thm:clique free simple}, $H$ has surplus $$\Omega_r(d^{r-1}/n^{r-3})\geq \Omega_r\left(m^{\frac{(r-1)^2}{2r-1}-\frac{r(r-3)}{2r-1}}\right)=\Omega_r(m^{\frac{r+1}{2r-1}}),$$ which is again the desired bound.
	\endproof

\section{Concluding remarks} \label{sec:concluding remarks}

\begin{itemize}
\item 
We proved that for any odd $r\geq 3$, $\sp(m,C_r)=\Theta_r(m^{(r+1)/(r+2)})$. The same function has been shown to be a lower bound when $r$ is even, but a matching upper bound is only known for $r\in \{4,6,10\}$ (see~\cite{AKS:05}). A related open problem is to construct graphs with $m$ edges and girth at least $r+1$ which have surplus $O_r(m^{(r+1)/(r+2)})$. Such examples are only known for $r=3$ and $r=4$ (see \cite{alon:96} and \cite{ABKS:03}).

\item
We obtained sharp bounds for the surplus of regular graphs as a function of the order, degree and number of triangles. As we remarked after Theorem~\ref{thm:triangledependence}, it is easy to find examples for which the statement of the theorem is not true when $G$ is not $d$-regular, but has average degree $d$. In light of Theorems~\ref{thm:triangle simple} and~\ref{thm:triangle square threshold}, it would be interesting to see if  this result can be extended to irregular graphs in a meaningful way.

\item
We also proved that any $K_r$-free graph with $m$ edges has surplus $\Omega_r(m^{\frac{1}{2}+\frac{3}{4r-2}})$. Arguably, the main open problem is to decide whether there exists a positive absolute constant~$\eps$ such that any $K_r$-free graph with $m$ edges has surplus $\Omega_r(m^{1/2+\eps})$. It seems that the bound $\Omega_r(m^{\frac{1}{2}+\frac{3}{4r-2}})$ is the best our methods can give. We think that any improvement on it, even just beating the exponent $5/7$ in the $r=4$ case, would be interesting.

\item
Another natural infinite sequence of graphs to consider is the set of complete bipartite graphs $K_{r,r}$. Forbidding bipartite subgraphs puts more restrictions on the host graph, so it is potentially easier to show that there exists an absolute constant $\eps>0$ such that any $K_{r,r}$-free graph with $m$ edges has surplus $\Omega_r(m^{1/2+\eps})$. Nevertheless, it is likely that this requires new ideas. See~\cite{AKS:05} for the correct exponent for $K_{r,s}$ when $r\in\Set{2,3}$, and a conjecture for the general case.

\item
Conlon, Fox, Kwan and Sudakov~\cite{CFKS:19} proved analogues of Edwards's result for hypergraphs. It would be interesting to further investigate the MaxCut problem in this setting and to obtain the tight bounds for the surplus.

\item
Finally, we remark that if Conjecture~\ref{conj:degeneracy} and the conjecture of Alon, Bollob\'as, Krivelevich and Sudakov stating $\sp(m,K_r)\geq c_rm^{3/4+\eps_r}$ are not true, then a possible way of disproving them is by finding a $K_r$-free strongly regular graph with many triangles. To be more precise, by Lemma~\ref{lemma:eigenvalue computation}, a strongly regular graph with degree $d$ and $d^3/6+\omega(nd^{3/2})$ triangles has smallest eigenvalue $|\lambda_{\min}|=o(\sqrt{d})$, which implies that it has surplus $o(n\sqrt{d})$. If a $K_r$-free graph exists (for some fixed $r$) with these properties, it disproves Conjecture~\ref{conj:degeneracy}.
\end{itemize}

\providecommand{\bysame}{\leavevmode\hbox to3em{\hrulefill}\thinspace}
\providecommand{\MR}{\relax\ifhmode\unskip\space\fi MR }
\providecommand{\MRhref}[2]{%
  \href{http://www.ams.org/mathscinet-getitem?mr=#1}{#2}
}
\providecommand{\href}[2]{#2}

\begin{appendices}

\section{The smallest eigenvalue of strongly regular graphs}

\begin{lemma} \label{lemma:eigenvalue computation}
    Let $G$ be a strongly regular graph with $n$ vertices, degree $d\leq 0.99n$ and $d^3/6+s$ triangles. Let $\lambda_{\min}$ be the smallest eigenvalue of $G$.
    \begin{itemize}
        \item If $s<-nd^{3/2}$, then $|\lambda_{\min}|= O(\frac{|s|}{nd})$.
        \item If $-nd^{3/2}\leq s\leq nd^{3/2}$, then $|\lambda_{\min}|= O(d^{1/2})$.
        \item If $s>nd^{3/2}$, then $|\lambda_{\min}|=O(\frac{nd^2}{s})$.
    \end{itemize}
\end{lemma}

\begin{proof}
We can assume that $n$ is sufficiently large, otherwise the statement is clear. Recall that
\begin{equation}
    \lambda_{\min} = \frac{1}{2}\left(\eta-\mu - \sqrt{(\eta-\mu)^2 + 4(d-\mu)} \right), \label{eqn:eigenvalue of strongly regular}
\end{equation}
where $\eta$ is the number of common neighbours of adjacent pairs and $\mu$ is the number of common neighbours of non-adjacent pairs. Double-counting the number of (not necessarily induced) copies of $K_{1,2}$ in $G$, we get
\begin{equation}
    n\binom{d}{2}=\frac{nd}{2}\eta+\left(\binom{n}{2}-\frac{nd}{2}\right)\mu. \label{eqn:eta and mu}
\end{equation}
Let $\delta=\eta-\frac{d^2}{n}$ and $\beta=\mu-\frac{d^2}{n}$. After a straightforward algebraic manipulation of (\ref{eqn:eta and mu}), we get
$$\beta=-\frac{d}{n-d-1}\delta-\frac{d(n-d)}{n(n-d-1)}.$$
Using $d\leq 0.99n$, we see that if $|\delta|\leq 6d^{1/2}$, then $|\beta|=O(d^{1/2})$; if $\delta<-6d^{1/2}$, then $\delta-\beta=-\Theta(|\delta|)$; and if $\delta>6d^{1/2}$, then $\delta-\beta=\Theta(\delta)$.

Now if $s<-nd^{3/2}$, then by (\ref{eqn:eta}), we have $\delta=\frac{6s}{nd}<-6d^{1/2}$, so $\eta-\mu=\delta-\beta=-\Theta(|\delta|)=-\Theta(\frac{|s|}{nd})$. Thus, (\ref{eqn:eigenvalue of strongly regular}) gives $\lambda_{\min}=-\Theta(\frac{|s|}{nd})$.

If $-nd^{3/2}\leq s\leq nd^{3/2}$, then $|\delta|=|\frac{6s}{nd}|\leq 6d^{1/2}$, so $|\eta-\mu|=|\delta-\beta|=O(d^{1/2})$. Thus, (\ref{eqn:eigenvalue of strongly regular}) gives that $|\lambda_{\min}|=O(d^{1/2})$.

Finally, if $s>nd^{3/2}$, then $\delta=\frac{6s}{nd}>6d^{1/2}$, so $\eta-\mu=\delta-\beta=\Theta(\delta)$. Note that $$\sqrt{(\eta-\mu)^2 + 4(d-\mu)}\leq \eta-\mu+\frac{2(d-\mu)}{\eta-\mu},$$ so by (\ref{eqn:eigenvalue of strongly regular}) we have $\lambda_{\min}\geq -\frac{d-\mu}{\eta-\mu}\geq -\frac{d}{\eta-\mu}$. Thus, $|\lambda_{\min}|= O(\frac{d}{\delta})=O(\frac{nd^2}{s})$.
\end{proof}

\section{Proofs for Section~\ref{sec:few triangles}}

Before proving Lemma~\ref{lemma:compare x and y} and Lemma \ref{lemma:sp with deltaplus}, we make a few computations which will be used in the proofs. Assume that $d\leq n/2$ and $0<\gamma\leq 1/10$, like in the setting of these lemmas. Recall that $a=\frac{\sqrt{d}}{n-d}$, which is at most  $\frac{2\sqrt{d}}{n}$ since $d\leq n/2$. Again using $d\leq n/2$, we have $a\leq \frac{1}{\sqrt{d}}$. In particular, $\frac{1}{d}+\frac{2a}{\sqrt{d}}+a^2\leq \frac{4}{d}$. 

Note also that if vectors $\vx^v$ and $\vy^v$ are defined as at the beginning of Section \ref{sec:few triangles}, then for every $v\in V(G)$, $\|\vx^v\|^2=\|\vy^v\|^2=(1+\gamma a)^2+d\frac{\gamma^2}{d}+(n-d-1)\gamma^2 a^2$. Write $q$ for this value and note that $1\leq q\leq (1+\gamma \frac{1}{\sqrt{d}})^2+\gamma^2+(n-d-1)\gamma^2\frac{d}{(n-d)^2}\leq 2$.

\lateproof{Lemma~\ref{lemma:compare x and y}}
For a vertex $v\in V(G)$, write $S(v)$ for the set of vertices $u\in N(v)$ for which $d(u,v)>20d^2/n$. Also write $F(v)=V(G)\setminus N(v)$.
First note that for any $uv\in E(G)$, we have
\begin{align*}
|\langle \vx^u,\vx^v\rangle| &\leq \sum_{w\in V(G)}|\vx^u_w||\vx^v_w| \le  2(1+\gamma a)\frac{\gamma}{\sqrt{d}} +  d\frac{\gamma^2}{d}+(n-d)\gamma^2a^2 + 2d \gamma^2\frac{a}{\sqrt{d}} \\ &\le 2\frac{\gamma}{\sqrt{d}}  + \gamma^2 \left(\frac{2}{n-d}+1+ \frac{d}{n-d}+2\frac{d}{n-d} \right) \leq 2\gamma + 6\gamma^2 \leq 1/2.
\end{align*}
The same estimate holds for $|\langle \vy^u,\vy^v\rangle| $. On the other hand, we have $\|\vx^v\|,\|\vy^v\|\geq 1$ for all~$v$. Hence, for any $uv\in E(G)$, $\left|\frac{\langle \vx^u,\vx^v\rangle}{\|\vx^u\| \|\vx^v\|}\right|\leq 1/2$ and $\left|\frac{\langle \vy^u,\vy^v\rangle}{\|\vy^u\| \|\vy^v\|}\right|\leq 1/2$.

The derivative of $\arcsin(x)$ is $\frac{1}{\sqrt{1-x^2}}$ which is always at least $1$, and for $-1/2\leq x\leq 1/2$ it is at most $2$. Thus, by the mean value theorem, we have
\begin{equation}
    \arcsin\left(\frac{\langle \vy^u,\vy^v\rangle}{\|\vy^u\| \|\vy^v\|}\right)-\arcsin\left(\frac{\langle \vx^u,\vx^v\rangle}{\|\vx^u\| \|\vx^v\|}\right)=\alpha \left(\frac{\langle \vy^u,\vy^v\rangle}{\|\vy^u\| \|\vy^v\|}-\frac{\langle \vx^u,\vx^v\rangle}{\|\vx^u\| \|\vx^v\|}\right) \label{eqn:arcsindiff}
\end{equation}
for some random variable $1\leq \alpha\leq 2$ ($\alpha$ depends on the choices of signs in $\vy^u$ and $\vy^v$).

Write $S(v)_+$ for the set of those $u\in S(v)$ for which $\vy^v_u=\gamma/\sqrt{d}$.
Now if $uv\in E(G)$, then
\begin{align*}
    \langle \vy^u,\vy^v\rangle - \langle \vx^u,\vx^v\rangle
    &=\mathbbm{1}_{u\in S(v)_+}\frac{2\gamma}{\sqrt{d}}+\mathbbm{1}_{v\in S(u)_+}\frac{2\gamma}{\sqrt{d}}+|(S(u)_+\cap F(v))\cup (S(v)_+\cap F(u))|\frac{2\gamma^2 a}{\sqrt{d}} \\
    &-|(S(u)_+\cap N(v))\cup (S(v)_+\cap N(u))\setminus (S(u)_+\cap S(v)_+)|\frac{2\gamma^2}{d}.
\end{align*}
Hence, using that $1\leq \alpha\leq 2$,
\begin{align*}
    \alpha\left(\langle \vy^u,\vy^v\rangle - \langle \vx^u,\vx^v\rangle \right)
    &\leq \mathbbm{1}_{u\in S(v)_+}\frac{4\gamma}{\sqrt{d}}+\mathbbm{1}_{v\in S(u)_+}\frac{4\gamma}{\sqrt{d}}+|(S(u)_+\cap F(v))\cup (S(v)_+\cap F(u))|\frac{4\gamma^2 a}{\sqrt{d}} \\
    &-|(S(u)_+\cap N(v))\cup (S(v)_+\cap N(u))\setminus (S(u)_+\cap S(v)_+)|\frac{2\gamma^2}{d}.
\end{align*}
If $w\in S(u)$, then $\bP(w\in S(u)_+)=1/2$, so
\begin{align}
    \expn{\alpha\left(\langle \vy^u,\vy^v\rangle - \langle \vx^u,\vx^v\rangle \right)}
    &\leq \mathbbm{1}_{u\in S(v)}\frac{2\gamma}{\sqrt{d}}+\mathbbm{1}_{v\in S(u)}\frac{2\gamma}{\sqrt{d}}+|(S(u)\cap F(v))\cup (S(v)\cap F(u))|\frac{2\gamma^2 a}{\sqrt{d}} \nonumber \\
    &-|(S(u)\cap N(v))\cup (S(v)\cap N(u))|\frac{\gamma^2}{d} \nonumber \\
    &\leq \frac{4\gamma}{\sqrt{d}}\mathbbm{1}_{d(u,v)>20d^2/n}+\frac{2\gamma^2 a}{\sqrt{d}}(|S(u)|+|S(v)|) \nonumber \\
    &-\frac{\gamma^2}{2d}(|S(u)\cap N(v)|+ |S(v)\cap N(u)|). \label{eqn:alphatimesdiff}
\end{align}
Recall that $\|\vx^v\|^2=\|\vy^v\|^2=q$ for every $v$, where $1\leq q\leq 2$. Therefore, by equations (\ref{eqn:arcsindiff}) and (\ref{eqn:alphatimesdiff}), we have
\begin{align}
    (\ast)&:=\expn{\sum_{uv\in E(G)} \left(\arcsin \frac{\langle \vy^u,\vy^v\rangle}{\|\vy^u\| \|\vy^v\|}-\arcsin \frac{\langle \vx^u,\vx^v\rangle}{\|\vx^u\| \|\vx^v\|}\right)} \nonumber \\
    &\leq \frac{4\gamma}{\sqrt{d}}\sum_{uv\in E(G)} \mathbbm{1}_{d(u,v)>20d^2/n}
    +\frac{2\gamma^2 a}{\sqrt{d}}\sum_{uv\in E(G)} (|S(u)|+|S(v)|) \nonumber \\
    &-\frac{1}{2}\cdot \frac{\gamma^2}{2d}\sum_{uv\in E(G)} (|S(u)\cap N(v)|+|S(v)\cap N(u)|). \label{eqn:star}
\end{align}
Now note that
\begin{align}
    \sum_{uv\in E(G)} (|S(u)\cap N(v)|+|S(v)\cap N(u)|)
    &= |\{(u,v,w)\in V(G)^3: uv,vw,wu\in E(G), d(u,w)>20d^2/n\}| \nonumber \\
    &=2\sum_{\substack{uw\in E(G): \\ d(u,w)>20d^2/n}} d(u,w). \label{eqn:a.4}
\end{align}
On the other hand,
\begin{equation}
    \sum_{uv\in E(G)} (|S(u)|+|S(v)|)=\sum_{u\in V(G)} d|S(u)|=\sum_{\substack{uv\in E(G): \\ d(u,v)>20d^2/n}} 2d. \label{eqn:sumofs}
\end{equation}
Finally, as the number of triangles in $G$ is at most $d^3/3$, the number of $uv\in E(G)$ with $d(u,v)>20d^2/n$ is at most $\frac{nd}{20}$, so
\begin{equation}
    \sum_{uv\in E(G)} \mathbbm{1}_{d(u,v)>20d^2/n} \leq \frac{nd}{20}. \label{eqn:sumofindicator}
\end{equation}
Thus, plugging (\ref{eqn:a.4}), (\ref{eqn:sumofs}) and (\ref{eqn:sumofindicator}) into (\ref{eqn:star}), we get
\begin{align*}
    (\ast)
    &\leq \gamma \frac{4n\sqrt{d}}{20}-\gamma^2\left(\sum_{\substack{uv\in E(G): \\ d(u,v)>20d^2/n}} \left(\frac{d(u,v)}{2d}-4\sqrt{d}a\right)\right) \leq \frac{\gamma n\sqrt{d}}{4}-\gamma^2\left(\sum_{\substack{uv\in E(G): \\ d(u,v)>20d^2/n}} \frac{d(u,v)}{10d}\right),
\end{align*}
where in the second inequality we used that $a\leq \frac{2\sqrt{d}}{n}$.
\endproof

\lateproof{Lemma~\ref{lemma:sp with deltaplus}}
Define the vectors $\vx^v$ as before. Looking at the Taylor series of $\arcsin(x)$, one can see that $\arcsin(x)\leq x+10x^3$ holds for all $0\leq x\leq 1$. On the other hand, $\arcsin(x)\leq x$ for $-1\leq x<0$. Thus, by (\ref{eqn:innerproduct}),
\begin{align*}
    \arcsin \frac{\langle \vx^u,\vx^v\rangle}{\|\vx^u\|\|\vx^v\|}
    &\leq \frac{\langle \vx^u,\vx^v\rangle}{\|\vx^u\|\|\vx^v\|}+10\left(\frac{1}{\|\vx^u\|\|\vx^v\|}\gamma^2 \left(\frac{1}{d}+\frac{2a}{\sqrt{d}}+a^2\right)\delta_+(u,v)\right)^3 \\
    &\leq \frac{\langle \vx^u,\vx^v\rangle}{\|\vx^u\|\|\vx^v\|}+\gamma^2\frac{\delta_+(u,v)^3}{10d^3},
\end{align*}
where in the second inequality we used $\|\vx^u\|,\|\vx^v\|\geq 1$, $\frac{1}{d}+\frac{2a}{\sqrt{d}}+a^2\leq \frac{4}{d}$ and $\gamma\leq 1/10$. Summing over all $uv\in E(G)$ and using that $q=\|\vx^v\|^2$ for every $v\in V(G)$, we get
\begin{align*}\sum_{uv\in E(G)} \arcsin \frac{\langle \vx^u,\vx^v\rangle}{\|\vx^u\|\|\vx^v\|}
&\leq \frac{1}{q}\sum_{uv\in E(G)} \langle \vx^u,\vx^v\rangle +\gamma^2\sum_{uv\in E(G)}\frac{\delta_+(u,v)^3}{10d^3}.
\end{align*}
By (\ref{eqn:innerproduct}),
\begin{align*}
    \sum_{uv\in E(G)} \langle \vx^u,\vx^v\rangle
    &=-\gamma n\sqrt{d}+\gamma^2\left(\frac{1}{d}+\frac{2a}{\sqrt{d}}+a^2\right)\sum_{uv\in E(G)}\delta(u,v) \\
    &=-\gamma n\sqrt{d}+\gamma^2\left(\frac{1}{d}+\frac{2a}{\sqrt{d}}+a^2\right)3s=-\gamma n\sqrt{d}+\Theta\left(\frac{\gamma^2 s}{d}\right)
\end{align*}
since $\sum_{uv\in E(G)} \delta(u,v)=\sum_{uv\in E(G)} d(u,v)-\frac{nd}{2}d^2/n=3t(G)-d^3/2=3s$.
Using $1\leq q\leq 2$, we get
$$\sum_{uv\in E(G)} \arcsin \frac{\langle \vx^u,\vx^v\rangle}{\|\vx^u\|\|\vx^v\|}
\leq -\frac{\gamma n\sqrt{d}}{2}+\Theta\left(\frac{\gamma^2 s}{d}\right)+\gamma^2\sum_{uv\in E(G)}\frac{\delta_+(u,v)^3}{10d^3}.$$
Substituting this into (\ref{eqn:sp with pi}) and using that if $d(u,v)>d^2/n$, then $\delta_+(u,v)\leq d(u,v)\leq d(u)=d$,
\begin{align*}
    \sp(G)
    &\geq -\frac{1}{\pi}\left(-\frac{\gamma n\sqrt{d}}{4}+\Theta\left(\frac{\gamma^2 s}{d}\right)+\gamma^2\sum_{uv\in E(G)} \frac{\delta_+(u,v)^3}{10d^3}-\gamma^2 \sum_{\substack{uv\in E(G): \\ d(u,v)>20d^2/n}} \frac{d(u,v)}{10d}\right) \\
    &\geq -\frac{1}{\pi}\left(-\frac{\gamma n\sqrt{d}}{4}+\Theta\left(\frac{\gamma^2 s}{d}\right)+\gamma^2 \sum_{\substack{uv\in E(G): \\ d(u,v)\leq 20d^2/n}} \frac{\delta_+(u,v)^3}{10d^3}\right),
\end{align*}
which completes the proof of the lemma.
\endproof

\end{appendices}

\end{document}